\documentclass[12pt,reqno]{amsart}
  \usepackage[all,2cell]{xy}
  \usepackage{amsfonts}
  \usepackage{amsthm}
  \usepackage{amsmath}
  \usepackage{amssymb}
 \numberwithin{equation}{section}
 \usepackage{color}
\usepackage[colorlinks]{hyperref}
\hypersetup{urlcolor=blue,linkcolor=blue,citecolor=blue}
\usepackage{enumerate,xspace}

\CompileMatrices

\DeclareMathOperator{\im}{im}

\renewcommand{\phi}{\varphi}
\renewcommand{\epsilon}{\varepsilon}

\newcommand{\ola}{\overleftarrow}
\newcommand{\ora}{\overrightarrow}

\newcommand{\F}{\ensuremath{\Fsk}\xspace}
\newcommand{\Fsk}{\textnormal{\bf Fsk}\xspace}

\newcommand{\Fpm}{\ensuremath{\F_{pm}}\xspace}
\newcommand{\Fpmfat}{\ensuremath{\F^{fat}_{pm}}\xspace}
\newcommand{\Fl}{\ensuremath{\F_{\lambda}}\xspace}
\newcommand{\Flfat}{\ensuremath{\F^{fat}_{\lambda}}\xspace}
\newcommand{\Fla}{\ensuremath{\F_{\lambda\alpha}}\xspace}
\newcommand{\Ul}{\ensuremath{U_{\lambda}}\xspace}
\newcommand{\Ula}{\ensuremath{U_{\lambda\alpha}}\xspace}
\newcommand{\Flafat}{\ensuremath{\F^{fat}_{\lambda\alpha}}\xspace}
\newcommand{\Far}{\ensuremath{\F_{\alpha\rho}}\xspace}
\newcommand{\Farfat}{\ensuremath{\F^{fat}_{\alpha\rho}}\xspace}

\newcommand{\C}{\ensuremath{\mathcal C}\xspace}

\renewcommand{\O}{\ensuremath{\mathcal O}\xspace}

\renewcommand{\S}{\ensuremath{\mathcal T}\xspace}

\newcommand{\XX}{\ensuremath{\mathbb X}\xspace}
\newcommand{\NN}{\ensuremath{\mathbb N}\xspace}

\newcommand{\Cat}{\textsf{Cat}\xspace}

\newcommand{\dbot}{\ensuremath{\mathbf{\Delta}_\bot}\xspace}

\newcommand{\op}{\ensuremath{^{\textnormal{op}}}}

\newcommand{\oprev}{\ensuremath{^{\textnormal{oprev}}}}

\newcommand{\ord}[1]{\ensuremath{\mathbf{#1}}}

\newcommand{\Tam}{\mathrm{Tam}}

\newcommand{\lra}{\longrightarrow}

\newcommand{\Lra}{\Longrightarrow}

\newcommand{\ox}{\otimes}
\newcommand{\x}{\times}

  \newtheorem{proposition}{Proposition}[section]
  \newtheorem{lemma}[proposition]{Lemma}
  
  \newtheorem{corollary}[proposition]{Corollary}
  \newtheorem{theorem}[proposition]{Theorem}

  \theoremstyle{definition}

  \newtheorem{example}[proposition]{Example}

  \theoremstyle{remark}
  \newtheorem{remark}[proposition]{Remark}

  \newcounter{c}
  
  \newcommand{\etyk}[1]{\vspace{-7.4mm}$$\begin{equation}\Label{#1}
  \addtocounter{c}{1}}
  \renewcommand{\]}{\ifnum \value{c}=1 $$\else \end{equation}\fi}
\begin{document}

 \title{Triangulations, orientals, and skew monoidal categories}

\author{Stephen Lack}
\address{Department of Mathematics, Macquarie University, NSW 2109, Australia}
\email{steve.lack@mq.edu.au}

\author{Ross Street}
\address{Department of Mathematics, Macquarie University, NSW 2109, Australia}
\email{ross.street@mq.edu.au}

\thanks{Both authors gratefully acknowledge the support of the Australian Research Council Discovery Grant DP1094883; Lack acknowledges with equal gratitude the support of an Australian Research Council Future Fellowship}

\date{\today}

\begin{abstract}
A concrete model of the free skew-monoidal category 
$\mathrm{Fsk}$ on a single generating object is obtained. 
The situation is clubbable in the sense of G.M. Kelly, 
so this allows a description of the free skew-monoidal category on any category.
As the objects of $\mathrm{Fsk}$ are meaningfully bracketed words in 
the skew unit $I$ and the generating object $X$, it is necessary to
examine bracketings and to find the appropriate kinds of morphisms between them. 
This leads us to relationships between triangulations of polygons,
the Tamari lattice, left and right bracketing functions, and the orientals. 
A consequence of our description of $\mathrm{Fsk}$ is a coherence theorem 
asserting the existence of a strictly structure-preserving faithful functor 
$\mathrm{Fsk} \lra \Delta_{\bot}$ where $\Delta_{\bot}$ is the skew-monoidal category of finite non-empty ordinals and first-element-and-order-preserving functions. This in turn provides a complete solution to the word problem for skew monoidal categories.     
\end{abstract}
  
\maketitle

%%%%%%%%%%%%%%%%%%%%%%%%%%%%%%%  INTRODUCTION  %%%%%%%%%%%%%%%%%%%%%%%%%%%%%%

\section{Introduction}

Counting the number of triangulations of a 
convex polygon is a famous problem, a brief 
history of which can be found in 
\cite[page 212]{Stanley-EC2}. It seems that the problem
is due to Euler, who proposed it to Segner. Segner
gave a recurrence relation for the solution, and 
Euler gave the formula appearing on the left
of the following equation
$$ \frac21 \cdot\frac63 \cdot\frac{10}{5}\cdot \ldots \cdot\frac{4n-2}{n+1}=
\frac1{n+1}\frac{(2n)!}{n!n
!} =\frac1{n+1}\binom{2n}{n}$$
% $$ \frac21 \cdot\frac63 \cdot\frac{10}{5}\cdot \ldots \cdot\frac{4n-10}{n-1}=
% \frac1{n-1}\frac{(2n-4)!}{(n-2)!(n-2)!} =
% \frac1{n-1}\binom{2(n-2)}{n-2}$$
for the number of triangulations of a convex polygon with $n+2$ vertices. 

This can easily be transformed to the expressions on the right, whose values are now known as the Catalan numbers, and it seems to have been Catalan \cite{Catalan} who realized the equivalence between triangulations of a polygon and bracketings. The following diagram illustrates how a triangulation of a 6-gon provides a bracketing
for a 5-fold product.
$$
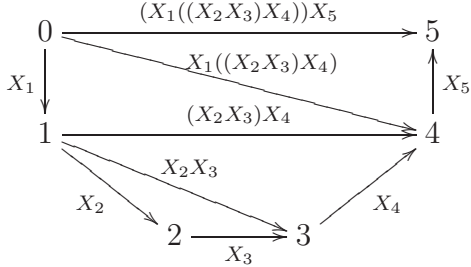
\xymatrix @C3pc {
0 \ar[d]_{X_1} \ar[drrr]^{X_1((X_2X_3)X_4)} \ar[rrr]^{(X_1((X_2X_3)X_4))X_5}  &&& 5 \\
1 \ar[dr]_{X_2} \ar[rrr]^{(X_2X_3)X_4} \ar[drr]^{X_2X_3} &&& 
4 \ar[u]_{X_5} \\
& 2 \ar[r]_{X_3} & 3 \ar[ur]_{X_4} 
}$$
In the case of an associative multiplication, of course there is only one product of an ordered sequence of terms, but the combinatorics of such bracketings becomes significant in the context of non-associative multiplications.

Tamari \cite{TamariThesis} 
considered a partial order on the set
of all such bracketings, where for bracketed
words $U$, $V$, and $W$ we have $(UV)W\le U(VW)$,
and where if $V\le V'$ then $UV\le UV'$ and
$VW\le V'W$.  The poset $\Tam_n$ of all such
bracketings of an $n$-fold product is in 
fact a lattice: this was proved in \cite{TamariLattice},
but a more transparent proof was found in 
\cite{TamariHuang}, using a combinatorial description
of bracketings similar to the one we shall use
below. See \cite{TamariFestschrift} for many articles
related to Tamari's work, including its connections
to the associahedra of Stasheff \cite{StasheffAssociahedra}.

Mac~Lane introduced the notion of monoidal category
\cite{MacLane-monoidal}, which involves a functorial
product, generally called the ``tensor product'', which 
need not be associative in the literal sense, but is 
associative up to natural isomorphism. Similarly there
is a ``unit object'', which need not satisfy the usual
unit laws in the literal sense, but does satisfy them
up to natural isomorphism. These associativity and
unit isomorphisms are required to satisfy five 
compatibility conditions, and Mac~Lane showed that 
these five conditions imply, in a precise sense, that
all diagrams built up using only these  ``structure 
isomorphisms''
must commute. The fact that all diagrams commuted,
in this sense, was summarized by saying that the 
structure of monoidal category was {\em coherent};
it is then a consequence that any monoidal category
is monoidally equivalent to a {\em strict} monoidal
category, in which the structure isomorphisms are
in fact identities.

As interest turned from monoidal categories to other
structures borne by categories, in which the 
``all diagrams commute''
condition did not hold, the focus came to be on
determining which diagrams did commute. A
deeper understanding \cite{Kelly-AbstractApproachCoherence}
of these coherence questions
came when they were seen to be part of the problem
of determining the free structure of the given type
on any category.  In many of the structures under 
interest it was seen that the free structure on any 
category could be obtained from the free structure 
on a single generating object (that is, on the one-object
discrete category 1). In these cases the structure
was said to be ``clubbable''.

Various weakenings of the notion of monoidal category have been studied, including weakenings obtained by dropping the requirement that the maps expressing associative and unit laws be invertible. Once invertibility is dropped, a particular choice of the direction of the maps needs to be made. Once such choice has been studied recently by   Szlach\'anyi \cite{Szlachanyi-skew} under the name of {\em skew monoidal category}. The crucial insight of Szlach\'anyi was that this structure provides a significant simplification of the notion of bialgebroid (or $\times_R$-bialgebra): see \cite{Takeuchi-bialgebra,Lu-QuantumGroupoids,Xu-QuantumGroupoids} for the origins of these notions, or the survey article \cite{Bohm-HopfAlgebroids} for an overview of the relations between them and the many applications they have found. In the papers \cite{skew,skew-EH} we have developed the connections between skew monoidal categories and quantum categories \cite{DayStreet-QuantumCats}. Perhaps surprisingly, it is also possible \cite{Street-SkewClosed} to do enriched category theory over a skew-monoidal base; we intend to develop this further elsewhere. 

The structure of skew monoidal category is clubbable, in the above sense, and so the free skew monoidal category on an arbitrary category exists, and can be described in terms of the free skew monoidal category on \ord1. Our main goal is to provide an explicit model for the free skew monoidal category \Fsk on \ord1. We do this in Theorem~\ref{thm:Fsk}. Unlike the situation of monoidal categories, it is not the case in \Fsk that all diagrams commute. What we do have is a faithful functor from \Fsk to the simplex category $\mathbf{\Delta}$: see Corollary~\ref{cor:faithful}. This means that in order to determine whether two expressions built up, using the operations of tensor and composition, out of the structure maps for skew monoidal categories agree, it suffices to check whether they agree in $\mathbf{\Delta}$. This is the sense in which we claim to have solved the word problem for skew monoidal categories. 

In March 2012, Kornel Szlach\'anyi told us he had started thinking about coherence in skew-monoidal categories. In June 2012, he told us he had proved that the homs were finite in the free skew-monoidal category generated by a set of objects. Soon we became hooked on the problem and this paper is the result.

Our secondary goal is to describe the connections between the Tamari lattices and the {\em orientals} of \cite{orientals}. The orientals come from the area of higher category theory. They are (strict) $\omega$-categories which are freely generated by a simplex, and can be used to define the nerve of an $\omega$-category. They arose originally in connection with non-abelian cohomology. A cubical version of the orientals was constructed with varying formalities by Iain Aitchison \cite{Aitchison-cubes}, the second author  \cite{Street-ParityComplexes}, and Michael Johnson \cite{Johnson-thesis}.

As was observed without proof in \cite{orientals}, triangulations of polygons appear as certain 2-cells in the orientals, and so a connection between orientals and the Tamari lattices is to be expected; here we make one such connection precise. This connection only uses a small fragment of the structure of the orientals. The paper \cite{KapranovVoevodsky-pasting} suggests further connections, and uses higher structure in the orientals to define {\em higher Stasheff posets}, which have also been studied under the name {\em higher Tamari posets} \cite{TamariFestschrift}.

A further interesting connection is mentioned in \cite{KapranovVoevodsky-pasting}, namely to the weak Bruhat order on the symmetric groups. It turns out (see also \cite{LodayRonco}) that the Tamari posets are quotients of the weak Bruhat orders; for example, the 5-element Tamari poset $\Tam_4$ is a quotient of the 6-element poset arising from the weak Bruhat order on the symmetric group $S_3$. The higher Bruhat orders of \cite{ManinSchechtman} can also be seen to arise out of cubical versions of the orientals, such as referred to above: see~\cite{KapranovVoevodsky-pasting}.

We now outline the contents of the paper. In Section~\ref{sect:axioms}, we recall the definition of skew monoidal category, as well as making precise what we mean by free skew monoidal category. In Section~\ref{sect:simplicial}, we recall some basic facts and notation about the simplex category $\mathbf{\Delta}$, and how it relates to our problem. In Section~\ref{sect:overview} we sketch briefly several different ways to encode bracketings, and describe how to translate between them. One of these ways to encode bracketings is provided by the ``bracketing functions'', which were central to \cite{TamariHuang}; we describe these in more detail in Section~\ref{sect:bracketings}, and prove a few key facts about them which will be needed later.  Then in Section~\ref{sect:orientals}, we give a full proof of the equivalence between bracketing functions and certain 2-cells in the orientals. We also show how to describe the order relation between bracketings in terms of 3-cells and 4-cells in the orientals. In the remainder of the paper we focus on our main problem of describing the free skew monoidal category on one object. We work towards this by first considering various simpler structures than skew monoidal category, and describing the corresponding free objects. Thus in Section~\ref{sect:a} we consider just a skew associative tensor product (with no unit), and in Section~\ref{sect:l} we consider an arbitrary tensor product (no associativity) with a left skew unit, while in Section~\ref{sect:la} we combine these two, to get a skew associative tensor product with a compatible left skew unit. In Section~\ref{sect:r} we describe how to dualize all of this and so obtain right skew units; while in Section~\ref{sect:lar} we come to the full structure of skew monoidal category. Finally in Section~\ref{sect:club} we recall from \cite{Kelly-AbstractApproachCoherence} enough of the theory of clubs to allows us to describe the free skew monoidal category on an arbitrary category. 

\section{Axioms, non-axioms, and freeness}
\label{sect:axioms}

A {\em (left) skew-monoidal category} is a category $\C$ equipped with an object $I$ (called the {\em unit} or {\em skew unit}),
a functor $\otimes : \C \times \C \to \C$ (called {\em tensor product}), and
natural families of {\em lax constraints} having the directions 
\begin{equation}
\alpha_{XYZ} : (X\otimes Y)\otimes Z \lra X\otimes (Y\otimes Z)
\end{equation}
\begin{equation}
\lambda_X : I\otimes X \lra X
\end{equation}
\begin{equation}
\rho_X :  X \lra X\otimes I
\end{equation}
subject to five conditions: 

\begin{equation}\label{pentagon}
\xymatrix{
& (W\otimes X)\otimes (Y\otimes Z) \ar[rd]^-{{\alpha_{W,X,Y\otimes Z}}}  & \\
((W\otimes X)\otimes Y)\otimes Z \ar[ru]^-{{\alpha_{W\otimes X,Y,Z}}} \ar[d]_-{{\alpha_{W,X,Y}} \otimes 1_Z} & & W\otimes (X\otimes (Y\otimes Z))  \\
(W\otimes (X\otimes Y))\otimes Z \ar[rr]_-{{\alpha_{W,X\otimes Y,Z}}} & & W\otimes ((X\otimes Y)\otimes Z) \ar[u]_-{1_W\otimes {\alpha_{X, Y,Z}}} }
\end{equation}

\begin{equation}\label{leftunit}
\xymatrix{
(I\otimes X)\otimes Y \ar[rd]_{\lambda_{X} \otimes 1_Y}\ar[rr]^{\alpha_{I,X,Y}}   && I\otimes (X\otimes Y) \ar[ld]^{\lambda_{X\otimes Y}} \\
& X\otimes Y  &
}
\end{equation}

\begin{equation}\label{midunit}
\xymatrix{
(X\otimes I)\otimes Y \ar[rr]^-{\alpha_{X,I,Y}} && X\otimes (I\otimes Y) \ar[d]^-{1_X\otimes \lambda_Y} \\
X\otimes Y \ar[u]^-{\rho_X \otimes 1_Y} \ar[rr]_-{1_{X\otimes Y}} && X\otimes Y}
\end{equation}

\begin{equation}\label{rightunit}
\xymatrix{
(X\otimes Y)\otimes I \ar[rr]^{\alpha_{I,X,Y}}   && X\otimes (Y\otimes I)  \\
& X\otimes Y \ar[lu]^{\rho_{X \otimes Y}}  \ar[ru]_{1_X\otimes \rho_Y}&
}
\end{equation}

\begin{equation}\label{unitunit}
\xymatrix{
I \ar[rd]_{\rho_I}\ar[rr]^{1_I}   && I  \\
& I\otimes I \ar[ru]_{\lambda_I} & .
}
\end{equation}
Of course a monoidal category is precisely a skew monoidal category in which the lax constraints are invertible.

We shall sometimes save space by omitting the tensor product symbol, writing $XY$ for $X\ox Y$; likewise, we shall sometimes omit the subscripts on the natural transformations $\alpha$, $\lambda$, and $\rho$.

Notice that we obtain idempotents
$$\epsilon^{\ell}_{X,Y} : (XI)Y \stackrel{\alpha} \lra X (I Y)
\stackrel{1\lambda} \lra X\otimes Y
\stackrel{\rho 1 } \lra (XI) Y \ ,$$
$$\epsilon^{r}_{X,Y} : X (I Y) \stackrel{1\lambda} \lra X\otimes Y
\stackrel{\rho 1} \lra (X I) Y \stackrel{\alpha} \lra X (I Y)$$
and
$$\epsilon_0 : I  I \stackrel{\lambda} \lra I \stackrel{\rho}\lra I I $$
% $$\epsilon^{\ell}_{X,Y} : (X\otimes I) \otimes Y \stackrel{\alpha_{X,I,Y}} \lra X\otimes (I \otimes Y)
% \stackrel{1_X \otimes \lambda_X} \lra X\otimes Y
% \stackrel{\rho_X \otimes 1_Y} \lra (X\otimes I)\otimes Y \ ,$$
% $$\epsilon^{r}_{X,Y} : X\otimes (I \otimes Y) \stackrel{1_X \otimes \lambda_X} \lra X\otimes Y
% \stackrel{\rho_X \otimes 1_Y} \lra (X\otimes I)\otimes Y \stackrel{\alpha_{X,I,Y}} \lra X\otimes (I \otimes Y)$$
% and
% $$\epsilon_0 : I\otimes I \stackrel{\lambda_I} \lra I \stackrel{\rho_I}\lra I\otimes I $$
which are not identities in general. Moreover, 
$$\alpha_{X,I,Y} : ((X\otimes I) \otimes Y, \epsilon^{\ell}_{X,Y}) \lra (X\otimes (I \otimes Y), \epsilon^{r}_{X,Y})$$
is a morphism of idempotents.

As explained in the introduction, our main goal is to give an explicit model for the free skew monoidal category on \ord 1. We have now given the precise definition of skew monoidal category, but we have not yet explained what we mean by the free skew monoidal category on a category \XX. 

We mean by this a skew monoidal category $\Fsk(\XX)$ equipped with a functor $X\colon \XX\to\Fsk(\XX)$ satisfying the following universal property. For any skew monoidal category \C and any functor $C\colon\XX\to\C$, there is a unique functor $F\colon\Fsk(\XX)\to\C$ which strictly preserves the skew monoidal structure and whose composite with $X$ is $C$.

We shall primarily be interested in the case where \XX is the terminal category \ord 1, in which case we write \Fsk for $\Fsk(\ord 1)$. Then $X$ is just an object of \Fsk
and $C$ is just an object of \C, and we require, for each $C$, a unique structure-preserving functor $F\colon\Fsk\to\C$ sending $X$ to $C$.

There is also a weaker, ``bicategorical'', meaning of free. This would involve functors from \Fsk to \C which preserve the structure only up to suitably coherent isomorphism, which send $G$ to an object isomorphic to $C$, and which are unique only up to a unique isomorphism. Whereas the free structures in the strict sense of the previous paragraph are determined up to isomorphism, these bicategorically free structures are determined only up to equivalence. A discussion of these matters can be found in \cite{JoyalStreet-braided}.

\section{Simplicial matters}
\label{sect:simplicial}

In this section we recall some standard material about the simplex category $\mathbf{\Delta}$, before using it to give an example of a skew monoidal category that will play an important role in this paper.

Recall the algebraists' simplicial category $\mathbf{\Delta}$.
The objects are the finite ordinals $$\mathbf{n} = \{0,1,\dots ,n-1\} \ .$$ There are two conflicting conventions for naming the objects of $\mathbf{\Delta}$. In algebraic contexts, and especially where the operation of ordinal sum is significant, the objects are generally named as we have done above; in topological contexts, where the objects are being thought of as simplexes rather than ordinals, and the dimension of these simplexes is important, the notation $[n-1]$ is common for what we are calling $\ord n$.

The morphisms are order-preserving functions $\xi : \mathbf{m} \lra \mathbf{n}$. We shall sometimes regard $\mathbf{\Delta}$ as a 2-category: for morphisms $\xi,\xi'\colon\ord m\to\ord n$ there is a 2-cell from $\xi$ to $\xi'$ just when $\xi(i)\le\xi'(i)$ for all $i\in\ord m$; we then write $\xi\le\xi'$. 

If $0\le i\le n$, we write $\delta_i\colon\ord n\to\ord{n+1}$ for the unique order-preserving injection whose image does not contain $i$; explicitly, $\delta_i(k)$ is equal to $k$ if $k<i$ and $k+1$ otherwise. Similarly, if $0\le i\le n-1$, we write $\sigma_i\colon\ord{n+1}\to\ord n$ for the unique order-preserving surjection which identifies $i$ and $i+1$; explicitly, $\sigma_i(k)$ is equal to $k$ if $k\le i$ and $k-1$ otherwise. 

As is well-known, the category $\mathbf{\Delta}$ is generated by these $\delta_i$ and $\sigma_i$, subject to certain relations \cite{CWM}. We wish to think of these relations as the following (directed) rewrite rules, which imply the existence of normal forms. 
\begin{align*}
\delta_i\delta_j  &\to \delta_{j+1}\delta_i   \tag{$i\le j$} \\
\sigma_i \sigma_{j+1}&\to \sigma_j\sigma_i  \tag{$i\le j$} \\
\sigma_j \delta_i &\to \delta_i\sigma_{j-1} \tag{$i<j$} \\
\sigma_j \delta_i &\to 1 \tag{$i=j,j+1$} \\
\sigma_j \delta_i &\to \delta_{i-1}\sigma_j \tag{$i>j+1$} 
\end{align*}
It follows that every morphism in $\mathbf{\Delta}$ can be written uniquely in the form 
$$\delta_{i_s}\ldots\delta_{i_2}\delta_{i_1}\sigma_{j_r}\ldots\sigma_{j_2}\sigma_{j_1},$$ 
with $j_1\le j_2\le\ldots\le j_s$ and $i_1<i_2<\ldots<i_r$. This is not the same normal form as given in \cite{CWM}, but is the one we shall need in what follows. In any case, we do get the same (unique) factorization into an order-preserving surjection $\sigma$ followed by an order-preserving injection~$\delta$.

For each morphism $\xi\colon\ord m\to\ord n$ in \ord\Delta, we put
$$\xi^{\ell} = \{i\in \mathbf{m-1} : \xi (i) = \xi (i+1)\}$$
and
$$\xi^{r} = \{j \in \mathbf{n} : j \notin \mathrm{im} \xi \} \ .$$
Notice that $\xi^{\ell} =\varnothing$ means that $\xi$ is injective and $\xi^{r} =\varnothing$ means that $\xi$ is surjective. Furthermore, in the factorization $\xi=\delta\sigma$ we have $\sigma^{\ell} = \xi^{\ell}$, $\sigma^{r} = \varnothing$, $\delta^{\ell} = \varnothing$, and $\delta^{r} = \xi_r$.  

Once again, it is well known \cite{Lawvere:ordinal-sum,CWM} that $\mathbf{\Delta}$ is the free strict monoidal category containing a monoid. This means that $\mathbf{\Delta}$ contains a monoid (the ``generic monoid''), and for any strict monoidal category \C containing a monoid, there is a unique strict monoidal functor from $\mathbf{\Delta}$ to \C sending the generic monoid to the given monoid in \C. 
The unit object is $\mathbf{0}$ and tensor product is ordinal sum: $\mathbf{m}\otimes \mathbf{n} = \mathbf{m+n}$, $\xi\otimes \zeta = \xi \mathbf{+} \zeta$.
The generic monoid is $\mathbf{1}$ with as multiplication and unit the unique morphisms $\ord 2\to\ord 1$ and $\ord 0\to\ord 1$. 

Any non-empty finite ordinal \ord n has 0 as bottom element and $n-1$ as top element. We shall also use the symbols $\bot$ and $\top$ for the bottom and top elements of an ordinal. 

Let $\mathbf{\Delta}_{\bot}$ denote the subcategory of $\mathbf{\Delta}$ consisting of non-empty finite ordinals $\mathbf{n}$
and first-element-preserving functions; that is, the $\xi$ with $0\notin \xi_r$. We may obtain a presentation for \dbot from the previous presentation for \ord\Delta~ by removing the $\delta_0$ as generators, and any rewrite rules which contain them. 

\begin{remark}
Just as $\mathbf{\Delta}$ plays a fundamental role in the study of monads (and monoids), so too \dbot plays a fundamental role \cite{Lawvere:ordinal-sum} in the study of algebras for a monad. 
\end{remark}

The tensor product on 
$\mathbf{\Delta}$ restricts to $\mathbf{\Delta}_{\bot}$ 
but the unit object \ord0 of $\mathbf{\Delta}$ does not lie in \dbot; instead, we shall see that \ord1 serves as a skew unit for \dbot.
The resulting skew-monoidal category has associativity constraint an identity, while
$$\lambda_{\mathbf{n}} : \mathbf{1+n} \lra \mathbf{n}$$ is the surjection $\sigma_0$ with 
$(\lambda_{\mathbf{n}})^{\ell} = \{0\}$, and 
$$\rho_{\mathbf{n}} : \mathbf{n} \lra \mathbf{n+1}$$ is the injection $\delta_n$ with
$(\rho_{\mathbf{n}})^{r} = \{n\}$.

When we write \dbot in future, we mean it equipped with the above skew-monoidal structure. 

\begin{remark}
  This skew monoidal category \dbot is in fact the {\em initial strictly-associative skew monoidal category}, in the sense that if \C is any skew monoidal category in which the associativity maps $\alpha$ are identities, then there is a unique functor $\dbot\to\C$ which preserves all of the skew monoidal structure. 
\end{remark}

Since \dbot is a skew monoidal category, and \ord 1 is an object of \dbot, once we know that \Fsk is the free skew monoidal category on \ord 1, we shall know that there is a unique functor $\Fsk\to\dbot$ which sends the generator to $\ord 1$ and which strictly preserves the skew monoidal structure. In fact, we work in the reverse direction, first constructing a category \Fsk equipped with a faithful functor $\Fsk\to\dbot$, then showing that the skew monoidal structure lifts strictly through this functor, and finally showing that the resulting skew monoidal category \Fsk is free on \ord1.

We now recall a few further well-known facts about \ord\Delta. 
A morphism $\xi : \mathbf{m}\lra \mathbf{n}$ in \ord\Delta~ has a right adjoint $\xi^* : \mathbf{n}\lra \mathbf{m}$ as order-preserving functions if and only if $\xi$ preserves the bottom element; in other words, if and only if $\xi$ is in \dbot. The formula is
\begin{equation}\label{radj}
\xi^*(j) = \mathrm{max}\{i : \xi(i) \le j \} \ .
\end{equation}
Analogously (and dually), $\xi\colon\ord m\to \ord n$ has a right adjoint if and only if it preserves the top element, so that $\xi(m-1)=n-1$.  The formula is 
\begin{equation}\label{ladj}
\xi^!(j) = \mathrm{min}\{i : j \le \xi(i) \} \ .
\end{equation}

\begin{remark}
Of course, as with any adjunction we have $i\le \xi^*(j)$ if and only if $\xi(i)\le j$, and likewise $\xi^!(j)\le i$ if and only if $j\le \xi(i)$.  But because we are dealing here with adjunctions between totally ordered sets, there are some extra equivalences: $i<\xi(j)$ iff $\xi(j)\not\le i$ iff $j\not\le\xi^*(i)$ iff $\xi^*(i)<j$; and similarly $\xi(i)<j$ iff $j\not<\xi(i)$ iff $\xi^!(j)\not< i$ iff $i\le \xi^!(j)$. Thus when working with strict inequalities rather than inequalities the roles of left and right adjoints are in some sense reversed. 
\end{remark}

\begin{remark}
Corresponding to the unit of the adjunction, we have the inequality $i\le \xi^*\xi(i)$ for all $i$. To say that $i<\xi^*\xi(i)$ for a particular $i$ is equivalently to say that there is an $i'>i$ with $\xi(i')=\xi(i)$; clearly there will be such an $i'$ if and only if $\xi(i+1)=\xi(i)$; which is to say, in the notation introduced earlier, that $i\in\xi^\ell$. Thus $i<\xi^*\xi(i)$ if and only if $i\in\xi^\ell$. Similarly $\xi\xi^*(j)\le j$ is true for all $j$, and the inequality is strict just when $j\in\xi^r$.
\end{remark}

Let $\mathbf{\Delta}_{\top}$ denote the subcategory of $\mathbf{\Delta}$ consisting of non-empty finite ordinals $\mathbf{n}$
and last-element-preserving functions.
This becomes a right skew-monoidal category in the obvious way; in fact, in a dual way as reinforced by the next result, which appeared already in \cite{Lawvere:ordinal-sum}.

\begin{proposition}%\label{bottop} 
Formula \eqref{radj} defines an isomorphism of categories
$$(\mathbf{\Delta_{\bot}})^{\mathrm{op}} \cong \mathbf{\Delta_{\top}} \ .$$ 
If $\zeta : \mathbf{m}\lra \mathbf{n}$ in $\mathbf{\Delta}_{\top}$ is
the image of $\xi : \mathbf{n}\lra \mathbf{m}$ in $\mathbf{\Delta}_{\bot}$ under the isomorphism then 
$$ \zeta^r = \xi^{\ell}  \ \text{  and  } \  \zeta^{\ell} = \{i-1 : i\in \xi^r \}  \ .$$ 
\end{proposition}
\begin{proof}
A functor between finite ordinals preserves limits if and only if it preserves last element.
A functor between finite ordinals preserves colimits if and only if it preserves first element.
So the isomorphism follows from the adjoint functor theorem and the fact that the right adjoint of a functor is the left Kan extension of the identity along the functor. 
The second sentence follows using composition of adjunctions $\tau \dashv \sigma \dashv \delta$
where $\delta$ and $\tau$ are injective, $\sigma$ is surjective, $\delta^r = \sigma^{\ell} = \{i\}$ and $\tau^r = \{i+1\}$
for some natural number $i$.      
\end{proof}

The category $\mathbf{\Delta}$ is equivalent to the category of all finite totally ordered sets and order-preserving functions. It will sometimes be convenient to use totally ordered sets which are not ordinals; up to isomorphism, of course, this makes no difference.

\section{Overview of different approaches to bracketings}
\label{sect:overview}

Our goal is to discover a nice model for the free skew-monoidal category \Fsk generated by a single object $X$. 
The objects are clearly words in the letters $I$ and $X$, bracketed meaningfully pairwise, such as $(((X(IX))(XX))X)$, and so we shall need to have ways of working efficiently with such bracketings.

The bracketing given above  can be written as a triangulated polygon
\begin{equation}\label{heptagon1}
\xymatrix{
& 0 \ar[rr]{} \ar[rrrd]{} \ar[rddd]{}  \ar[ld]_{X}  & & 6  &  \\
1 \ar[rrdd]{} \ar[d]_{I} & & & & 5 \ar[lu]_{X} \\
2  \ar[rrd]_{X} & & & & 4  \ar[u]_{X} \\
& &  3 \ar[rruu]{}  \ar[rru]_{X} &
}
\end{equation}
where the edges $i\lra i+1$ are labelled by the symbols $I$ or $X$. This labelling can equivalently be specified by giving the subset $u\subseteq\ord m$ consisting of all $i\in\ord m$ for which the edge $i\to i+1$ is labelled by $X$.

In most of what follows the $I$/$X$ labelling will be straightforward to deal with; most of the interest will lie in the triangulation/bracketing. When we focus on this, the particular identity of the things being bracketed is not so important, but it is convenient to index them, as in $(((X_0(X_1X_2))(X_3X_4))X_5)$.

The triangles appearing in the triangulation can be represented as lists $(x_1x_2x_3)$, where $x_1<x_2<x_3$ are the objects; for example the large triangle in the centre would be $(035)$. For each $i\in\{1,2,3,4,5\}$, there is exactly one triangle $(x_1x_2x_3)$ with $x_2=i$. If we write $\ell(i-1)$ for $x_1$ and $r(i)+1$ for $x_3$, we obtain a pair of functions $\ell$ and $r$, and as we shall see, the whole triangulation can be recovered either from $\ell$ or from $r$. In the case of the triangulation above, for instance, we would have $\ell(2)=0$ and $r(3)=4$.

For a given bracketing of 6 terms, each of $X_0,X_1,\ldots,X_5$ is bracketed either to the left or the right. Of course $X_0$ will always be bracketed to the right and $X_5$ to the left. In our example, $X_2$ and $X_4$ are bracketed to the left, and $X_1$ and $X_3$ to the right. 

We can recognize this distinction in each of the different representations of a bracketing, as summarized in the following table, in which conditions appearing in the same column are equivalent.

\begin{figure}[h]
  \centering
\begin{tabular}[c]{|c|c|}
\hline 
  $X_i$ bracketed to left & $X_i$ bracketed to right \\
\hline
$\ell(i)<i$ & $\ell(i)=i$ \\
\hline
$r(i)=i$ & $r(i)>i$ \\
\hline
triangle $(j,i,i+1)$ & triangle $(i,i+1,k)$ \\
\hline
\end{tabular}    
\end{figure}

The other approach to bracketings that we shall use involves the orientals of \cite{orientals}, in which bracketings turn out to be certain 2-cells in an $n$-category.  We shall discuss this in Section~\ref{sect:orientals}. 

Yet another approach to bracketings, important in many contexts, but not used in this paper, involves binary trees.

\section{Bracketing functions and the Tamari lattice}
\label{sect:bracketings}

We shall now formalize the sort of ``bracketing functions'' which arise as the $\ell$ and $r$ of the previous section; these were studied previously in \cite{TamariHuang}. In this section we work abstractly with these, before turning in the next section to their connections with orientals. In fact we consider bracketing functions for an arbitrary finite non-empty totally ordered set, rather than one of the form $\{0,1,\ldots,m-1\}$.

Let $M$ be a finite non-empty totally ordered set. We write $\partial M$ for the subset consisting of the top and bottom elements $\top$ and $\bot$, and $\mathring M$ for the complement of $\partial M$ in $M$ (the interior of $M$). We shall sometimes use $i+1$ and $i-1$ to denote the successor and predecessor of an element $i$. We write $M\op$ for the set $M$ with the reverse ordering. 

A {\em left bracketing function} ({\em lbf\/}) is a function $\ell:M \lra M$ satisfying
\begin{itemize}
\item[(i)] $\ell(j) \le j$ for all $j\in M$;
\item[(ii)] $\ell(j)\le i \le j$ implies $\ell(j)\le \ell(i)$;
\item[(iii)] $\ell$ preserves the top element $\top$.
\end{itemize}
Some consequences are:
\begin{itemize}
\item[(iv)] $\ell$ preserves the bottom element $\bot$
\item[(v)] $\ell$ is idempotent.
\end{itemize}
As with any functions into a linearly ordered set, the lbf are ordered using value-wise order in $M$, so that $\ell\le \ell'$ just when $\ell(i)\le\ell'(i)$ for all $i\in M$. We write $\Tam_M$ for the resulting poset.

If $M\cong M'$ as posets, then clearly $\Tam_M\cong\Tam_{M'}$, and so in some sense little is lost if we replace $M$ by the unique ordinal which is order-isomorphic to $M$. It is convenient, however, to allow ourselves the extra flexibility of a general $M$. In the case $M=\ord m$, we write simply $\Tam_m$ for $\Tam_M$: this is the Tamari lattice \cite{TamariThesis,Tamari62, TamariHuang}. For the equivalence between left bracketing functions $\ord m\to\ord m$ and $m$-fold bracketings see \cite{Tamari62} and \cite{TamariHuang}.

A {\em right bracketing function} ({\em rbf\/}) on $M$ is an lbf on $M\op$. In terms of $M$, this is a function $r:M \lra M$ satisfying
\begin{itemize}
\item[(i)] $i \le r(i)$ for all $i\in M$;
\item[(ii)] $i\le j \le r(i)$ implies $r(j)\le r(i)$;
\item[(iii)] $r$ preserves the bottom element $\bot$.
\end{itemize}
Some consequences are:
\begin{itemize}
\item[(iv)] $r$ preserves the top element $\top$;
\item[(v)] $r$ is idempotent.
\end{itemize}

\begin{proposition}\label{bijbf}
The following equations determine a bijection between lbfs $\ell:M \lra M$ 
and rbfs $r:M \lra M$: for $i,j\in\mathring M$ %$0<i,j<m-1$,
$$r(i) = \mathrm{min} \{j : \ell(j) < i \le j \}$$
$$\ell(j) = \mathrm{max} \{i : i \le j < r(i) \} .$$
\end{proposition}
\begin{proof} 
Given an lbf $\ell$, let $r$ be defined as in the proposition 
together with the requirement that it should preserve top and bottom elements.
We need to see first that $r$ is an rbf.
As every $j$ with $\ell(j) < i \le j$ has $i \le j$, we obtain $i \le r(i)$.  
To prove property (ii) for $r$, suppose $i\le j\le r(i)$.
Then $\ell(k)<i \le k$ implies $j\le k$. If $r(i) < r(j)$ then there exists
$k$ with $\ell(k)< i\le k$ but not $\ell(k)< j\le k$. Since we know $j\le k$,
we must have $j \le \ell(k)$. Then we have the contradiction
$i\le j\le \ell(k) < i$. So $r(j) \le r(i)$ as required.

Now we need to show that we can recover $\ell$ from the second formula in the
proposition. We need to see that $\ell(j)$ is the biggest $i$ such that $i\le j$ and, if
$\ell(k) < i \le k$, then $j<k$. To see that $i = \ell(j)$ does have the property notice that
$\ell(j) \le j$ by (i) for $\ell$, and, by (ii) for $\ell$, $\ell(k) < \ell(j) \le k$ and 
$k\le j$ would imply $\ell(j)\le \ell(k)$, a contradiction. 
To see that $\ell(j)$ is the biggest such $i$, suppose we had
$\ell(j) < i$ with $i\le j$ and $\ell(k) < i \le k$ implying $j<k$. Then $j$ is such a $k$.
So $j<j$, a contradiction.

This proves half of the proposition. The other half follows by duality (replace $M$ by $M\op$).
\end{proof}

When $\ell$ and $r$ correspond under the bijection of Proposition~\ref{bijbf}
we put $\ora{\ell} = r$ and $\ell=\ola{r}$.

We emphasize that $\ell(\bot)=\bot=r(\bot)$, $\ell(\top)=\top=r(\top)$ and, for $i,j\in\mathring M$,
we have $\ell(j)<\top$ and $r(i)>\bot$. 
Yet we can have $\ell(j)=\bot$ 
(when $\{i : i \le j < r(i) \}=\varnothing$) or $r(i)=\top$ 
(when $\{j : \ell(j) < i \le j \}=\varnothing$).

\begin{proposition}\label{ordbf}
The bijection of Proposition~\ref{bijbf} is order preserving.
\end{proposition}
\begin{proof}
Assume $\ell \le \ell_1$. Then $\{ j : \ell_1 (j)< i\le j \} \subseteq \{j: \ell (j)< i \le j\} $.
So $\ora{\ell} \le \ora{\ell_1}$.
\end{proof}

Consequently, we may consider the elements of 
$\mathrm{Tam}_m$ to be rbfs rather than lbfs.

\begin{proposition}\label{fixbf}
Suppose $\ell$ and $r$ correspond under the bijection of Proposition~\ref{bijbf}.
\begin{enumerate}[(i)]
\item For $i\in\mathring M$, we have $\ell(i)=i$ if and only if $r(i) \neq i$.
\item For $i,j\in\mathring M$, if $\ell (j) = i$ then $r(i) \neq j$.
\item If $i\neq\bot$ and $r(i)\neq\top$ then $\ell r(i) \le \ell(i-1)$.
\item If $j\neq\top$ and $\ell(j)\neq\bot$ then $r(j+1) \le r\ell(j)$. 
\item[(v)] If $i\neq\bot$, $\ell(i-1)\neq\bot$, and $r(i)\neq\top$, then $\ell r(i) \neq \ell(i-1)$
implies $r(i) = r\ell(i-1)$.  
\end{enumerate}
\end{proposition}
\begin{proof}
For (i) we have $\ell(i)\neq i$ iff $\ell(i)<i$ iff $\ell(i)<i\le i$ iff $\ora{\ell}(i)\le i$
iff $\ora{\ell}(i) = i$.   

For (ii), if $i=\ell(j)$ then $j$ cannot be a $k$ with $\ell(k)<i \le k$ (let alone the first such).
So $\ora{\ell}(i) \neq j$. 

For (iii), since $r(i)<\top$, $r(i)$ is a $k$ satisfying $\ell(k) < i\le k$, and so $\ell r(i)<i$ and $\ell r(i)\le i-1$. But, by property (ii) for a lbf, $\ell r(i)\le i-1< i \le r(i)$ implies $\ell r(i)\le \ell (i-1)$. 

Then (iv) follows from (iii) by duality. 

Now to (v). Assume $i\neq\bot$, $\ell(i-1)\neq\bot$, $r(i)\neq\top$ and $\ell r(i) \neq \ell(i-1)$.
By (ii), $\ell r(i) < \ell(i-1)$. Now $r\ell(i-1)$ is the minimum $j$ with 
$\ell(j)<\ell(i-1)\le j$. Yet $\ell r(i) < \ell(i-1) \le r(i)$. So $r\ell(i-1)\le r(i)$.
But, by (iv), $r(i)\le r\ell(i-1)$. So $r(i)= r\ell(i-1)$. 
\end{proof}

\begin{proposition}[Huang-Tamari]
  If $(\ell_\gamma)_{\gamma\in\Gamma}$ is a family 
of lbf's in $\Tam_m$, then their pointwise join is again 
an lbf, and is therefore the join in $\Tam_m$.
\end{proposition}

\begin{proof}
The pointwise join $\ell$ is given by 
$\ell(j)=\max_{\gamma\in\Gamma} \ell_\gamma(j)$.
Since each $\ell_\gamma(j)\le j$, also 
$\max_{\gamma\in\Gamma}\ell_\gamma(j)\le j$.
Since each $\ell_\gamma(m-1)=m-1$ also 
$\max_{\gamma\in\Gamma}\ell_\gamma(m-1)=m-1$.
Finally, suppose that 
$\max_{\gamma\in\Gamma}\ell_\gamma(j)\le i\le j$.
Then for each $\gamma$ we have 
$\ell_\gamma(j)\le i\le j$, and so since $\ell_\gamma$
is an lbf we have
$\ell_\gamma(j)\le\ell_\gamma(i)$ (once again for all
$\gamma$). It then follows that 
$\max_{\gamma\in\Gamma}\ell_\gamma(j)\le\max_{\gamma\in\Gamma}\ell_\gamma(i)$.
\end{proof}

\begin{remark}
Since $\Tam_m$ is a finite poset with finite joins, it follows that $\Tam_m$ has finite meets, which gives Tamari's result that $\Tam_m$ is a lattice.  Notice, however, that the pointwise meet of a family of lbf morphisms need not be an lbf morphism, so meets are not constructed pointwise in general. But the dual result to the proposition shows that meets of rbf's are rbf's, and so we obtain the following description of meets in $\Tam_m$: given a family of elements $\ell_\gamma$ of $\Tam_m$, first form the corresponding rbf's $r_\gamma$, then their (pointwise) meet $r$, and then the corresponding lbf $\ell$.
\end{remark}

\begin{proposition}\label{prop:cob-for-lbf}
Suppose $\xi : M \lra N$ preserves order and top and bottom element
so that there are adjoints $\xi^! \dashv \xi \dashv \xi^*$.
If $\ell\colon M \lra M$ is an lbf then so is $\xi \ell \xi^*:N \lra N$. 
If $r:M \lra M$ is an rbf then so is $\xi r \xi^!:N \lra N$.
\end{proposition}
\begin{proof} 
Clearly $\ell_1=\xi \ell \xi^*$ preserves top element since all three factors do. Also $\ell_1(j) = \xi \ell \xi^*(j) \le \xi \xi^* (j) \le j$ since (i) holds for $\ell$ and $\xi$ preserves order. 
It remains to prove (ii) for $\ell_1$.
Now $\ell_1(j) \le i \le j$ implies $\ell\xi^*(j) \le \xi^*(i) \le \xi^*(j)$ which implies, using (ii) for $\ell$, that $\ell\xi^*(j) \le \ell\xi^*(i)$. 
However $\xi$ is order preserving, so $\ell_1(j)\le \ell_1(i)$.
The result for $r$ follows by duality.
\end{proof}

In particular, if $\sigma\colon M\to N$ is an order-preserving surjection then it has both adjoints, and so both these constructions are possible.

\begin{proposition}\label{prop:cob-lr}
  If $\sigma\colon M\to N$ is an order-preserving surjection between totally ordered sets, and $\ell\colon M\to M$ is an lbf, then we have $\sigma\ora{\ell}\sigma^!\le\ora{\sigma\ell\sigma^*}$. Furthermore, if $i<\sigma^*\sigma(i)$ implies $\sigma\ell(i)=\sigma(i)$, then we have an equality $\sigma\ora{\ell}\sigma^!=\ora{\sigma\ell\sigma^*}$.
\end{proposition}

\begin{proof}
Write $r$ for $\ora{\ell}$.
First we prove the inequality $\sigma r\sigma^!\le\ora{\sigma\ell\sigma^*}$.

We must show that $\sigma r\sigma^!(k)\le\ora{\sigma\ell\sigma^*}(k)$ for all $k$. If there is no $h$ with $\sigma\ell\sigma^*(h)<k\le h$, then $\ora{\sigma\ell\sigma^*}(k)$ will be the top element and there is nothing to prove. Suppose then that there is such an $h$. By adjointness we have $\ell\sigma^*(h)<\sigma^!(k)$; also $\sigma^!(k)\le\sigma^*(k)\le\sigma^*(h)$, giving $\ell\sigma^*(h)<\sigma^!(k)\le\sigma^*(h)$, and so $r\sigma^!(k)\le \sigma^*(h)$ by definition of $r=\ora{\ell}$. Using adjointness again gives $\sigma r\sigma^!(k)\le h$. But this will be true for all such $h$, and so in particular for the least such, namely $h=\ora{\sigma\ell\sigma^*}(k)$. This gives  $\sigma r\sigma^!(k)\le \ora{\sigma\ell\sigma^*}(k)$.

% On the other hand if there is no such $h$, then $\ora{\sigma\ell\sigma^*}(k)$ is the top element, and so $\sigma r\sigma^!(k)\le\ora{\sigma\ell\sigma^*)(k)$ must be true. 

Now consider the reverse inequality $\ora{\sigma\ell\sigma^*}(k)\le\sigma r\sigma^!(k)$. Once again, if there is no $h$ with $\ell(h)<\sigma^!(k)\le h$ then there is nothing to prove, so suppose that there is such an $h$, and consider the least such $h$, namely $r\sigma^!(k)$.
By adjointness we have $\sigma\ell(h)<k\le\sigma(h)$.

If $h=\sigma^*\sigma(h)$ then we have $\sigma\ell\sigma^*\sigma(h)<k\le\sigma(h)$, from which it follows that $\ora{\sigma\ell\sigma^*}(k)\le\sigma(h)$; but this is the required inequality $\ora{\sigma\ell\sigma^*}(k)\le\sigma r\sigma^!(k)$.

Suppose finally that $h<\sigma^*\sigma(h)$. By the additional hypothesis in the last sentence of the proposition we then have $\sigma\ell(h)=\sigma(h)$, and so $\sigma(h)<k\le\sigma(h)$, which is a contradiction. 
\end{proof}

Before leaving this section, we record one more result about bracketings which will be needed later. Recall from Section~\ref{sect:overview} the distinction, for a given bracketing $S\in\Tam_m$ and given $i\in\ord m$, between whether $X_i$ is bracketed to the left or to the right; equivalently whether $\ell(i)<i$ or $\ell(i)=i$.

\begin{proposition}\label{prop:wedge-ells}
  The set $\Tam^i_m$ of all $\ell\in\Tam_m$ for which $\ell(i)<i$ is down-closed, in the sense that if $\ell\in\Tam^i_m$ and $\ell'\le \ell$ then $\ell'\in\Tam_m$. Furthermore, $\Tam^i_m$ is closed in $\Tam_m$  under finite joins.
Dually,  the set $\Tam^{(i)}_m$ of all $\ell\in\Tam_m$ for which $\ell(i)=i$ is up-closed, in the sense that if $\ell\in\Tam^{(i)}_m$ and $\ell\le \ell'$ then $\ell'\in\Tam^{(i)}_m$. Furthermore, $\Tam^{(i)}_m$ is closed in $\Tam_m$  under finite meets.
\end{proposition}

\begin{proof}
If $\ell'\le \ell$ and $\ell(i)<i$ then $\ell'(i)\le\ell(i)<i$.
If $(\ell_\gamma)_{\gamma\in\Gamma}$ is a finite family of elements of $\Tam^i_m$, then their join $\ell$ in $\Tam^m$ is given by $\ell(j)=\max_{\gamma\in\Gamma}\ell_\gamma(j)$. Since $\ell_\gamma(i)<i$ for all $\gamma\in\Gamma$, also $\max_{\gamma\in\Gamma}\ell_\gamma(i)<i$, and so $\vee_{\gamma\in\Gamma}\ell_\gamma\in\Tam^i_m$.

For the dual case, use the fact that $\ell(i)=i$ if and only if $i<r(i)$, and that meets can be constructed pointwise using rbf's.
\end{proof}

\section{Bracketings, triangulations, and orientals}
\label{sect:orientals}

In this section we make precise the connection between bracketing functions, triangulations, and orientals.

We shall use the convention that each subset $x=\{x_0,x_1, \dots ,x_r\}$ of $\mathbf{m}\in \mathbf{\Delta}$ organizes itself as a list 
$x=(x_0x_1\dots x_r)$ with $x_0<x_1< \dots <x_r$. 
We may omit the curly brackets around some sets of such subsets and may also omit the commas between elements. 
For $0\le p \le r$, the {\em $p$-th face} $x\partial_p$ of 
$x$ is obtained by deleting $x_p$ from $x$.
The face is called {\em odd} or {\em even} according as $p$ is odd or even.
We write $x^-$ for the set of odd faces of $x$ and we write $x^+$ for the set of even faces. 
Because we start with an even $0$, the number of even faces is either equal to, or one greater than, the number of odd faces, depending on whether $m$ is odd or even.

%\red omit?
In particular, a {\em triangle} of \ord m is a set $x=(x_0x_1x_2)$ with {\em vertices}  $x_0<x_1<x_2$. We call $(x_0x_1)$ the {\em left leg}, $(x_1x_2)$ the {\em right leg}, and $(x_0x_2)$ the {\em long leg} of the triangle. Then $x^-$ consists of the long leg of $x$, while $x^+$ consists of the other two legs, which unsurprisingly are called the {\em short legs}. We also call $x_1$ the {\em middle vertex}.
%\black

If $H$ is a set of subsets of $\ord m$ of cardinality $r+1$, we define
\begin{eqnarray}
H^- = \{x^- : x\in H\} \ \text{  and  } \ H^+ = \{x^+ : x\in H\} \ . 
\end{eqnarray}
Also, we define
\begin{eqnarray}
H^{\mp} = H^- \smallsetminus H^+ \ \text{  and  } \ H^{\pm} = H^+ \smallsetminus H^- \ . 
\end{eqnarray}

For each $\xi : \mathbf{m} \lra \mathbf{n}$ in $\mathbf{\Delta}$, we write
$\xi x$ for $\{\xi(x_0)\xi(x_1) \dots \xi(x_r)\} \subseteq \mathbf{n}$. 
So for each $\xi : \mathbf{m} \lra \mathbf{n}$ in $\mathbf{\Delta}_{\bot}$
and $y=(y_0y_1\dots y_s) \subseteq \mathbf{n}$, we have
$$\xi^* y = \{\xi^*(y_0)\xi^*(y_1) \dots \xi^*(y_s)\} \subseteq \mathbf{m} \ .$$ 

Now we recall the {\em oriental} $\O_m$ \cite{orientals}. This is the free $m$-category on the $m$-simplex. However, we need no more than the underlying 4-category structure. The objects (= 0-cells) of $\O_m$ are the natural numbers $p$ with $0 \le p \le m$. A morphism (= 1-cell) $a : p \lra q$ can exist only if $p\le q$; then $a=(a_0 a_1\dots a_r)$ is a subset of $\mathbf{q}\smallsetminus \mathbf{p}$
with $a_0 =p$; we think of $a$ as the path 
$$p\stackrel{(pa_1)}\lra a_1\stackrel{(a_1a_2)}\lra a_2\lra \dots \stackrel{(a_rq)}\lra q \ .$$
According to \cite{orientals} we think of the morphism $a : p \lra q$ as a set of doublets 
$$a=(pa_1)(a_1a_2)\dots (a_rq) \ .$$  
We shall only need to use the morphism $0\stackrel{(0m)}\lra m$, which we call $b_m$, and the morphism $0\stackrel{(01)}\lra 1 \stackrel{(12)}\lra 2 \dots  \stackrel{(m-1,m)}\lra m$, which we call $e_m$.

% For our purposes, we are interested in the morphism $b_m : 0 \lra m$ which is $0\stackrel{(0m)}\lra m$ and
% the morphism $e_m : 0 \lra m$ which is the path $0\stackrel{(01)}\lra 1 \stackrel{(12)}\lra 2 \dots  \stackrel{(\overline{m-1}m)}\lra m$. 

Let $\S_m = \O_m (0,m)(b_m,e_m)$ as a 2-category. We shall examine this more explicitly. 

The objects of $\S_m$ are 2-cells $S:b_m \Lra e_m$ in $\O_m$.
A description of these, adapted from \cite{orientals}, is as follows: 
$S$ is a set of triangles $x=(x_0x_1x_2)$ of $\mathbf{m+1}$ 
satisfying the conditions:
\begin{itemize}
\item[(a)] if triangles $x$ and $y$ in $S$ share the same left leg, or the same right leg, or the same long leg, then they are equal;
% $x,y\in S$, if two of the three equations $x_0=y_0$, $x_1=y_1$, $x_2=y_2$ hold then $x=y$;
\item[(b)] for triangles $x$ and $y$ in $S$, if $x_1=y_0$ then $x_2\neq y_1$;
\item[(c)] $b_m = S^{\mp}$;
\item[(d)] $e_m = S^{\pm}$.
\end{itemize} 
Some consequences are:
\begin{itemize}
\item[(e)] the function $S\lra \{1,2,\dots ,m-1\}$, taking $x$ to its middle vertex $x_1$, is a bijection;
\item[(f)] for $x,y\in S$, if $x_0<y_1\le x_1$ then $x_0\le y_0$;
\item[(g)] for $x,y\in S$, $y_2\le x_1$ if and only if $x_0< y_1<x_1$.
\end{itemize} 
Injectivity of the function in (e) follows from the alternating position (AP) condition proved in \cite{orientals}; surjectivity is a simple induction related to the excision of extremals
algorithm in \cite{orientals}. 

These objects $S$ are in bijection with triangulations of the polygon with $m$ sides.
For example, the triangulation \eqref{heptagon1} of the heptagon has 
\begin{equation}
S=(013)(123)(035)(345)(056) \ .
\end{equation} 
It will be useful to have some notation, foreshadowed in Section~\ref{sect:overview}, involving the inverse $t_S$ of the bijection in (e) above:
\begin{eqnarray}\label{tlr}
t_S : \{1,2,\dots ,m-1\} \lra S \nonumber \\
t_S(i) = (\ell_S(i-1), i , r_S(i)+1) \ .
\end{eqnarray}
Notice that, in the process, we are defining functions 
\begin{eqnarray}
\ell_S : \mathbf{m-1} \lra \mathbf{m-1}
\end{eqnarray}
and 
\begin{eqnarray}\label{rightfns}
r_S: \{1,2,\dots ,m-1\} \lra \{1,2,\dots ,m-1\} \ .
\end{eqnarray}
However, by putting $\ell_S(m-1)=m-1$ and $r_S(0)=0$, we extend them to bottom-and-top-preserving functions
\begin{equation}\label{lr}
\ell_S, r_S : \mathbf{m} \lra \mathbf{m} \ .
\end{equation}
We put 
\begin{equation}
S_{\ell} = \{i : 0\le i<m-1, \ell_S(i)=i \}
\end{equation}
and
\begin{equation}
S_r = \{i :  0<i\le m-1, r_S(i)=i \} \ .
\end{equation}

\begin{proposition}\label{Sbijlbf} 
The formula
$$S=\{(\ell(i-1), i , \ora{\ell}(i)+1) : 0<i<m-1\}$$
establishes a bijection between lbf morphisms $\ell : \mathbf{m} \lra \mathbf{m}$ 
and  2-cells $S:b_m \Lra e_m$ in $\O_m$. 
\end{proposition}
\begin{proof}
Suppose $\ell$ is a lbf. Put $r=\ora{\ell}$. 
Define $S$ as in the proposition. 
We need to prove properties (a), (b), (c), (d) for $S$.
Property (e) is clear, so only one case for property (a) remains. 
Take $x=(\ell(i-1), i , r(i)+1)$
and $y=(\ell(j-1), j , r(j)+1)$ in $S$. We need to see that 
$\ell (i-1) = \ell (j-1)$ and $r(i)=r(j)$ imply $i=j$.
Assume $i<j$ since the hypotheses are symmetric.
From $r(i)=r(j)$, we see that any $k$ with
$\ell (k) < i \le k$ will have $r (j) \le k$. However, $j-1$ is just such
a $k$ since $\ell (j-1) = \ell (i-1) \le i-1 < i \le j-1$. So $r (j) \le j-1$.
But then $j \le r(j) \le j-1$ gives a contradiction.   

For property (b), take $x$ and $y$ in $S$ as above. 
This time we need to see that $x_1=y_0$ implies $x_2\neq y_1$;
that is, $i=\ell(j-1)$ implies $r(i)+1\neq j$. We do have
$j-1<m-1$ and $0<i=\ell(j-1)$, so, using (iv) of Proposition~\ref{fixbf},
we have $r(i)=r\ell (j-1) \ge r(j) \ge j > j-1$. So $r(i)+1\neq j$.  

Now we look at property (c). Since, by (i), $\ell(i-1) < i$, there must exist an 
$i$ with $\ell (i-1) = 0$.
Choose the largest such $i$. Then $\ell(k)<i\le k$ implies $\ell(k) \le i-1 <k$,
so, by (ii) for an lbf, $\ell(k) \le \ell(i-1)=0$. So $\ell(k)=0$. 
By maximality of $i$, $k\le i-1$, a contradiction. 
So the set of $k$, for which $r(i)$ is the minimum, is empty.
So $r(i) =m-1$. This shows that $b_m\subseteq S^-$.
Yet $b_m\cap S^+=\varnothing$ since there is no $x\in S$
with $x_0<0$ or $x_2>m$. 
So $b_m\subseteq S^{\mp}$.
Next we need to show that, for any $x\in S$, if $x^-\neq (0m)$ then $x^- \in S^+$.
Assume $x=(\ell(i-1), i , r(i)+1)$ and $0< \ell(i-1)$ (since the case
$r(i)+1<m$ will follow by left-right symmetry). 
We need to show that there exists either $y$ or $z$ in $S$ with
$y\partial_0 = (l(i-1),r(i)+1)$ or $z\partial_2 = (l(i-1),r(i)+1)$.
The last two conditions force $y=(\ell(j-1),j, r(j)+1)$ to have $j=\ell(i-1)$
and $z=(\ell(k-1),k, r(k)+1)$ to have $k=r(i)+1$. So we must see that,
with these choices for $j$ and $k$, either  $r(j)=r(i)$ or $\ell(k-1)=\ell(i-1)$.
Assume $\ell(k-1)\neq \ell(i-1)$; that is, $\ell r (i) \neq \ell(i-1)$. 
Then part (v) of Proposition~\ref{fixbf} implies $r(i)=r\ell(i-1)=r(j)$.

We also need to show $e_m = S^{\pm}$. Take $(i,i+1)\in e_m$. 
By part (i) of Proposition~\ref{fixbf}, either $i = r(i)$ or $\ell(i) =i$. 
So either $(i,i+1) = (\ell(i-1), i,r(i)+1)\partial_0$
or $(i,i+1) = (\ell(i), i+1,r(i+1)+1)\partial_2$. 
However we cannot have $(i,i+1) = (\ell(j-1), j,r(j)+1)\partial_1$
since $\ell(j-1)$ and $r(j)+1)$ have $j$ strictly between them 
and so cannot be consecutive.
So $e_m\subseteq S^{\pm}$. 
Now we prove $S^{\pm}\subseteq e_m$ by showing that $x\in S$
with $(x_1,x_2)\notin e_m$ or $(x_0,x_1)\notin e_m$ implies $x^+ \subseteq S^-$.
Suppose $x=(\ell(i-1), i , r(i)+1)$ and $(i , r(i)+1)\notin e_m$.
Then $i<r(i)$, so, by part (i) of Proposition~\ref{fixbf}, we have $\ell(i)=i$. 
Let $j$ be the largest such that $i\le j\le r(i)$ and $\ell(j)=i$.
By part (v) of Proposition~\ref{fixbf}, we have $r(j) = r(i)$.
So $x\partial_0 = (i , r(i)+1) = (\ell(j-1),j,r(j)+1)\partial_1 \in S^-$. 
The case $x\partial_2\notin e_m$ is similar. 

So far, we have defined a function $\mathrm{Tam}_m \lra \S_m$ which
is obviously injective. If we accept that the two sets have the same cardinality
(given by a Catalan number), the bijection follows. However, we will
describe the inverse function and complete the proof. 

Suppose $S:b_m \Lra e_m$ is a 2-cell. Using property (e) for $S$, 
define top-and-bottom-preserving functions 
$\ell = \ell_S$ and $r=r_S$ by \eqref{tlr} and \eqref{lr}. 
Then $\ell(i-1)< i < r(i)+1$ for $0<i<m-1$.
This gives property (i) for $\ell$ to be an lbf  (and for $r$ to be an rbf).
To prove property (ii) for $\ell$, take $\ell(j)\le i\le j<m$ and put
$x=(\ell(j),j+1,r(j+1)+1)$ and $y=(\ell(i),i+1,r(i+1)+1)$ which are in $S$.
Then $x_0<y_1\le x_1$. By property (f) for $S$, $x_0\le y_0$; so $\ell(j)\le \ell(i)$. 
It remains to show that $r = \ora{\ell}$. 
This breaks into two parts. First we must see that $r(i)<m-1$ implies 
$\ell r(i)<i\le r(i)$. For this, take $x=(\ell r(i),r(i)+1,r(r(i)+1)+1)$ and
$y=(\ell (i-1),i,r(i)+1)$ so that $x_1=y_2$. By the `only if' part of property (g) for $S$,
$x_0<y_1<y_2$, our desired result. The second part is minimality; that is, to show
$\ell(j)<i\le j$ implies $r(i)\le j$. Put $x=(\ell r(i),r(i)+1,r(r(i)+1)+1)$ and
$y=(\ell(i-1),i,r(i)+1)$ so that $x_0<y_1<x_1$. 
By the `if' part of property (g) for $S$, $y_2\le x_1$; that is, $r(i)\le j$.
\end{proof}

Now let us look at the morphisms of $\S_m$.
As 3-cells in $\O_m$, they are well defined in \cite{orientals} as sets $\theta$ of
four-element subsets $x=(x_0x_1x_2x_3)$ of $\mathbf{m+1}$ satisfying
well formedness and movement conditions; also see \cite{Street-ParityComplexes,Street-ParityComplexesCorrections}.
The indecomposable $\theta : S\lra T$ are precisely the singleton sets $\{ x\}$
such that
\begin{equation}\label{xmove}
T = (S\smallsetminus x^{-})\cup x^{+} \ ,
\end{equation}
where 
\begin{eqnarray*}
x^{-} = (x_0x_1x_2)(x_0x_2x_3) \ \text{ and  }  \ x^{+} = (x_0x_1x_3)(x_1x_2x_3) \ .
\end{eqnarray*}
So $x^{-}$ consists of the odd faces of $x$ and $x^{+}$ the
even faces in the sense of \cite{orientals}, although the notation is from \cite{Street-ParityComplexes}. 
Let us put
\begin{equation}
xS=T \ \text{   and   } \ S=Tx
\end{equation}
when \eqref{xmove} holds.
In pictures, this means that $S$ should contain the domain of
$x=(x_0x_1x_2x_3)$ in the diagram \eqref{indec3}
as part of its triangulation and $T$ should be obtained from $S$ by 
replacing that part by the codomain.

\begin{eqnarray}\label{indec3}
\xymatrix{
x_0 \ar[r]^-{} \ar[rd]^-{} \ar[d]_-{} & x_3 \\
x_1 \ar[r]_-{}  & x_2 \ar[u]^-{} }
\qquad
\xymatrix{
\stackrel{x} \lra}
\qquad
\xymatrix{
x_0 \ar[r]^-{} \ar[d]_-{} & x_3 \\
x_1 \ar[r]_-{} \ar[ru]^-{}  & x_2 \ar[u]^-{} }
\end{eqnarray}
In terms of bracketings, it means that $T$ is obtained from $S$ by moving just 
one set of brackets to the right:
$$((W_0((W_1W_2)W_3))W_4) \lra ((W_0(W_1(W_2W_3)))W_4)$$
or
$$(W_0(((W_1W_2)W_3)W_4)) \lra (W_0((W_1(W_2W_3))W_4))$$
where the $W_i$ are meaningfully binarily bracketed words in $I$ and $X$.

Now let us look at the 2-cells of $\S_m$.
As 4-cells in $\O_n$, they are well defined in \cite{orientals} as sets $a$ of
five-element subsets $x=(x_0x_1x_2x_3x_4)$ of $\mathbf{m+1}$ satisfying
well formedness and movement conditions.
The indecomposable $a : \theta \Lra \phi$ are precisely 
the singleton sets $\{ x\}$ such that
\begin{equation}
\phi = (\theta \smallsetminus x^{-})\cup x^{+} \ ,
\end{equation}
where 
$$x^{-} = (x_0x_1x_2x_4)(x_0x_2x_3x_4)$$ and 
$$x^{+} = (x_1x_2x_3x_4)(x_0x_1x_3x_4)(x_0x_1x_2x_3) \ .$$
In diagram~\eqref{O4}, we see the only non-identity 2-cell in $\S_4$.

\begin{equation}\label{O4}
\xymatrix{
(012)(023)(034) \ar[d]_{(0234)}^(0.8){\phantom{AAAAAAAA}}="1" \ar[rr]^{(0123)}  && (013)(123)(034) \ar[d]^{(0134)}_(0.8){\phantom{AAAAAAAA}}="2" \ar@{=>}"1";"2"^-{(01234)}
\\
(012)(024)(234) \ar[rd]_-{(0124)} && (014)(123)(134) \ar[ld]^-{(1234)} 
\\
& (014)(124)(234) 
}
\end{equation}

The following result is immediate on combining Proposition~\ref{Sbijlbf} with \cite{TamariHuang} and \cite{orientals}.

\begin{proposition}\label{TLatt} For 2-cells $S,T:b_n\Lra e_n : 0\lra n$ in the $n$-category 
$\O_n$, the following conditions are equivalent:
\begin{itemize}
\item[(i)] $\ell_S \le \ell_T$;
\item[(ii)] $r_S \le r_T$;
\item[(iii)] there exists a 3-cell $S\lra T$ in $\O_n$;
\item[(iv)] for all $x\in S$ and $y\in T$, if $x_1=y_1$ then $x_0\le y_0$; 
\item[(v)] for all $x\in S$ and $y\in T$, if $x_1=y_1$ then $x_2\le y_2$;
\item[(vi)] for all $x\in S$ and $y\in T$, if $x_1=y_1$ then $x_0\le y_0$ and $x_2\le y_2$.
\end{itemize}
\end{proposition}

So this means that the Tamari lattice $\mathrm{Tam}_m$ of the previous section is obtained from the 2-category
$\S_m$ by identifying all the 2-cells. That is, there is a 2-functor
$\S_m \lra \mathrm{Tam}_m$ (where the only 2-cells in 
$\mathrm{Tam}_m$ are the identities); and every 2-functor 
$\S_m \lra \C$ which takes all 2-cells in $\S_m$ to identities
factors uniquely through $\S_m \lra \mathrm{Tam}_m$. 

% If $\ell_S \le \ell_T$ in $\mathrm{Tam}_n$, we define the {\em distance} 
% from $S$ to $T$ by
% \begin{equation}
% \mathrm{d}(S,T) = \mathrm{min} \{\# \theta : \theta : S \lra T \text{ in } \S_n \} \ .
% \end{equation}
% Clearly, $\mathrm{d}(S,S) = 0$ and $\mathrm{d}(S,U) \le \mathrm{d}(S,T)+\mathrm{d}(T,U)$ for $S\le T\le U$. 
% As an example obtained by looking at diagram~\eqref{O4}, we have
% $$\mathrm{d}((012)(023)(034), (014)(124)(234)) = 2 \ .$$

\section{Skew associativity} 
\label{sect:a}

We start by considering simpler structures than that of skew monoidal category: define a {\em skew semimonoidal category} to be a category \C equipped with a functor $\ox\colon\C\x\C\to\C$ and a natural transformation 
$\alpha$ with components 
$$\xymatrix{(X\ox Y)\ox Z \ar[r]^{\alpha_{X,Y,Z}} & X\ox(Y\ox Z) }$$
which satisfies the pentagon condition \eqref{pentagon}.

  Let $\Tam$ be the poset whose objects consist of a positive integer $m$ and an element $S$ of the Tamari lattice $\Tam_m$, with $(m,S)\le(n,T)$ when $m=n$ and $S\le T$ in $\Tam_m$. There is a functor (order-preserving function) $\ox\colon\Tam\x\Tam\to\Tam$ defined by 
$$(m,S)\ox(n,T)=(m+n,S\star T)$$
where
% For $m,n > 0$, there is an order-preserving function 
% \begin{equation}
% \star : \mathrm{Tam}_m \times \mathrm{Tam}_n \lra \mathrm{Tam}_{m+n}
% \end{equation}
% defined by 
$$S\star T = S\cup \{(0,m,m+n)\} \cup (m+T)$$
and where $m+T = \{(m+i,m+j,m+k) : (i,j,k) \in T\}$.  
In other words,
\begin{equation*}
  t_{S\star T} (i) =
      \begin{cases}
         t_S(i) & \text{if } \  i < m  \\
         (0,m,m+n) & \text{if } \  i = m  \\
         (m,m,m)+t_T(i-m) & \text{if } \ i>m \ .
       \end{cases}
\end{equation*}
The existence of the 3-cell 
$$(0,m,m+n,m+n+k) : (S\star T)\star U \Lra S\star (T\star U)$$
in $\O_{m+n+k}$ shows we have lax associativity in the form
\begin{equation}\label{Tamlaxassoc}
(S\star T)\star U \le S\star (T\star U) 
\end{equation} 
and so that $\Tam$ becomes a skew semimonoidal category.

On the other hand there is no lax unit for $\star$. In particular,  $\mathrm{Tam}_1 = \{\varnothing\}$ is the terminal poset; 
its element is the identity 2-cell of $b_1 = e_1 : 0\lra 1$ in $\O_1$.
This element does not act as a lax unit since we have 
$$\varnothing \star S =  \{(0,1,1+m)\} \cup (1+S)$$ 
and 
$$S \star \varnothing = S\cup \{(0,m,m+1)\}$$
in $\mathrm{Tam}_{m+1}$; so these cannot be Tamari compared with 
$S\in \mathrm{Tam}_{m}$.  

We write $X$ for the object of $\Tam$  given by $1$ with the unique element of $\Tam_1$.

\begin{proposition}
  $\Tam$ is the free skew semimonoidal object on a single object, with $X$ as generator.
\end{proposition}

\begin{proof}
Let \C be a skew semimonoidal category and $C$ an object of \C. We have to show that there is a unique functor $F$ from $\Tam$ to \C which strictly preserves the skew monoidal structure and which sends $X$ to $C$.

It is clear that $F$ must send $(m,S)$ to the tensor product in \C of $m$ copies of $C$, bracketed according to $S$. An associativity morphism $(XY)Z\to X(YZ)$ in $\Tam$ is sent to the corresponding associativity morphism in \C. The indecomposable morphisms in $\Tam$ are obtained from the associativity morphisms by (repeatedly) tensoring on either side with identity morphisms, and so must be sent to the morphism in \C obtained by tensoring the corresponding associativity morphism by the corresponding identity morphisms. Finally a general morphism of $\Tam$ is a composite of indecomposable morphisms, and so must be sent to the corresponding composite in \C.

This description makes it clear that such an $F$ is unique, while existence reduces to the fact that any diagram in \C built up out of composites of tensors of associativity morphisms must commute. The analogous fact for monoidal categories is due to Mac~Lane \cite{MacLane-monoidal}, but an inspection of his proof shows that it works equally well for skew semimonoidal categories.
\end{proof}

\section{Left skew units}
\label{sect:l}

In the previous section we saw that the disjoint union $\Tam$ of
the Tamari posets is the free skew semimonoidal 
category on one object. 
This involved bracketings like $X((XX)X)$ of a single
object $X$. If we are to have a unit object $I$ then we will
need bracketings like $X((XI)(X(IX)))$ of $X$s and $I$s.

As a very basic initial step, we could consider the
structure consisting of a category \C equipped with
an arbitrary functor $\ox\colon\C\x\C\to\C$ 
and an arbitrary object $I$. For want of a better name,
we call this a pointed magmoidal category. 

Since this structure only 
involves functors of the form $\C^n\to \C$, and 
no natural transformations between them, the free
such structure on a discrete category will still be
discrete. In particular, the free pointed magmoidal category \Fpm on \ord 1 will be discrete. 

An object will consist of a non-empty finite ordinal \ord m, a subset $u\subseteq \ord m$, and an element $S\in\Tam_m$ of the Tamari lattice. As has already been anticipated, the cardinality  $m$ of $\ord m$ indicates that $(\ord m,u,S)$ is an $m$-fold product of copies of $X$ and $I$, bracketed according to $S$, while the subset $u$ indicates which of these factors are $X$s.
Sometimes it is useful to think of $u$ as the image of an injective order-preserving map $\partial_u$.

The tensor product $(\ord m,u,S)\ox(\ord n,v,T)$ has the form $(\ord m+\ord n,u\star v,S\star T)$. The product $S\star T$ was defined in Section~\ref{sect:a}; recall that 
$$\ell_{S\star T}(i) =
\begin{cases}
  \ell_S(i) & \text{if $i<m-1$} \\
0 & \text{if $i=m-1$} \\
m+\ell_T(i-m) & \text{if $i>m$.} 
\end{cases}$$
Finally $u\star v$ is defined by saying that $\partial_{u\star v}$ is the ordinal sum $\partial_u+\partial_v$ of the maps $\partial_u$ and $\partial_v$. Thus $u\star v$ contains all $i\in u$ as well as all $m+j$ with $j\in v$.

The unit object $I$ is $(1,\emptyset,*)$, where $*$ is the unique bracketing in $\Tam_1$, and the generator is $(1,1,*)$.

% We call this category \Fpm. 
The universal property is clear; we record it as:

\begin{proposition}
  The free pointed magmoidal category \Fpm
on one object
is the discrete category in which an object consists 
of a non-empty finite ordinal $\ord m$, 
a subset $u\subseteq \ord m$,
and an element $S\in\Tam_m$ of the Tamari lattice.
\end{proposition}

\begin{remark}\label{rmk:fat}
\Fpm is equivalent to the following category, which has multiple isomorphic copies of each object. An object is a finite non-empty totally ordered set $M$ equipped with a subset $u\subseteq M$ and an lbf $\ell\colon M\to M$. A morphism $(M,u,\ell)\to(M',u',\ell')$ is an invertible order-preserving function $M\cong M'$ which respects $u$ and $\ell$ in the evident sense. We call this category \Fpmfat. % We shall sometimes use ``$+1$'' and ``$-1$'' to denote successors and predecessors in totally-ordered sets.
\end{remark}

We shall gradually introduce further structure, 
culminating in the structure of skew monoidal 
category. But all of this extra structure will involve
natural transformations; the functors $\C^n\to \C$
will not change. Thus the objects of the corresponding
free structures will remain the same as the objects
of \Fpm.

As a first step, we define a {\em skew-left-unital magmoidal
category} to be a category \C equipped with a functor
$\ox\colon\C\x\C\to\C$, an object $I$, and a 
natural transformation $\lambda\colon I\ox-\to 1_\C$.

\subsection{Shrink morphisms}

First we recall that any order-preserving surjection $\sigma\colon M\to N$ between finite non-empty totally-ordered sets has both a right adjoint $\sigma^*$ and a left adjoint $\sigma^!$, given by 
\begin{align*}
  \sigma^*(j) &= \max\{i\mid \sigma(i)\le j\} \\
  \sigma^!(j) &= \min \{i\mid j\le \sigma(i)\}
\end{align*}
and in fact since $\sigma$ is surjective we may replace the defining inequalities $\sigma(i)\le j$ and $j\le\sigma(i)$ by equalities.

Let $(M,u,S)$ and $(N,v,T)$ be objects of \Fpmfat. Define a {\em shrink morphism} from $(M,u,S)$ to 
$(N,v,T)$ to be a surjective order-preserving map 
$\sigma\colon M\to N$ for which 
\begin{enumerate}[(a)]
\item $\sigma^*$ induces a bijection from $v$ to $u$
\item if $\sigma(j)=\sigma(j+1)$ then 
$\sigma\ell_S(j)=\sigma(j)$
\item $\ell_T=\sigma \ell_S \sigma^*$
\end{enumerate}

\begin{remark}\label{rmk:u}
Since $\sigma\sigma^*=1$, if condition (a) holds then the inverse $u\to v$ to $\sigma^*$ must be given by $\sigma$ itself. But we cannot merely replace (a) by the condition that $\sigma$ induce a bijection from $u$ to $v$ since the inverse could still fail to be $\sigma^*$. What would be needed in addition is that each element of $u$ is terminal in its $\sigma$-fibre; in other words, that each element of $u$ is {\em not} in $\sigma^\ell$.
\end{remark}

\begin{remark}\label{rmk:shrink-r}
  In light of Proposition~\ref{prop:cob-for-lbf}, if (b) holds then (c) is equivalent to $r_T=\sigma r_S \sigma^!$.
  % \begin{enumerate}
  % \item[(d)] $r_T=\sigma r_S \sigma^!$ 
  % \end{enumerate}
\end{remark}

\begin{proposition}
  Composites of shrink morphisms are 
shrink morphisms.
\end{proposition}

\begin{proof}
Let $\sigma\colon(M,u,S)\to(N,v,T)$ and 
$\tau\colon(N,v,T)\to(P,w,U)$ be shrink morphisms, and consider $\tau\sigma$. Conditions (a) and
(c) in the definition of shrink morphism clearly 
hold for $\tau\sigma$, but we should check (b).

Suppose then that $\tau\sigma(j)=\tau\sigma(j+1)$.
Either $\sigma(j)=\sigma(j+1)$, or $\sigma(j+1)=\sigma(j)+1$ and $\tau(\sigma(j)+1)=\tau\sigma(j)$.

If $\sigma(j+1)=\sigma(j)$, then $\sigma\ell_S(j)=\sigma(j)$ since $\sigma$ is a shrink morphism. It
clearly follows that $\tau\sigma\ell_S(j)=\tau\sigma(j)$.

Otherwise, we have $\tau\ell_T\sigma(j)=\tau\sigma(j)$,
since $\tau$ is a shrink morphism, and we 
have $\sigma^*\sigma(j)=j$, since $\sigma(j)<\sigma(j+1)$. Thus
$\tau\sigma\ell_S(j)=\tau\sigma\ell_S\sigma^*\sigma(j)=\tau\ell_T\sigma(j)=\tau\sigma(j)$ as required.
\end{proof}

We write \Flfat for the category of shrink morphisms, and \Fl for the full subcategory consisting of those object $(M,u,S)$ for which $M$ is an ordinal $\ord m$.

\subsection{Existence of a shrink morphism }

\begin{lemma}\label{lemma:sigma-existence}
  If $(M,u,S)$ is an object of \Flfat, an order-preserving 
surjection $\sigma\colon M\to N$ defines a shrink
morphism with domain $(M,u,S)$ if and only if 
\begin{enumerate}[(a)]
\item $\sigma^\ell\cap u=\emptyset$ 
\item if  $j\in\sigma^\ell$ then $\sigma\ell_S(j)=\sigma(j)$.
\end{enumerate}
The codomain is then $(N,\sigma u,T)$, where
$\ell_T=\sigma\ell_S\sigma^*$.
%, or equivalently $r_T=\sigma r_S\sigma^!$.
\end{lemma}

\begin{proof}
Condition (b) is copied directly from the definition. Condition (a) must hold by Remark~\ref{rmk:u}. Suppose then that the two conditions do hold. The fact that condition (a) in the definition holds follows again by Remark~\ref{rmk:u}, while condition (c) holds by definition of $T$.
\end{proof}

\begin{remark}
  In the special case where $\sigma$ has the 
form $\sigma_j\colon \ord{n+1}\to \ord n$, then (a) says
that $j\notin u$, while (b) says that $\ell_S(j)=j$.
\end{remark}
 
\begin{proposition}\label{prop:CS-factorization}
Suppose that $\sigma\colon(\ord m,u,S)\to(\ord n,v,T)$ is a 
shrink morphism with $m\neq n$, and let 
$j$ be the least element of $\ord m$ with $\sigma(j)=\sigma(j+1)$; in other words,  $j$ is least element of $\sigma^\ell$.  Then there is a unique factorization of the 
shrink morphism $\sigma$ as a shrink morphism $\sigma_j$ followed by a shrink morphism $\sigma'$. 
\end{proposition}

\begin{proof}
Since $\sigma(j+1)=\sigma(j)$ and $\sigma$ is 
a shrink morphism, we have
$\sigma\ell_S(j)=\sigma(j)$. Since $\ell_S(j)\le j$
and by the minimality of $j$, this implies 
$\ell_S(j)=j$. Also $j\in\sigma^\ell$ so $j\notin u$.
It follows that $\sigma_j$ defines a shrink morphism $(\ord m,u,S)\to(\ord{m-1},u',S')$.

We have a factorization $\sigma=\sigma'\sigma_j$
for a unique order-preserving map $\sigma'$. It
remains to show that $\sigma'$ is a shrink morphism $(\ord{m-1},u',S')\to(\ord n,v,T)$. 

Condition (a) for $\sigma'$ follows immediately 
from the corresponding conditions for $\sigma$
and $\sigma_j$.

For (c), we have 
$\ell_T = \sigma\ell_S\sigma^* = \sigma'\sigma_j \ell_S\sigma^*_j (\sigma')^* = \sigma' \ell_{S'} (\sigma')^*$
using the definition of $\ell_{S'}$ and the fact
that $\sigma$ is a shrink morphism.

Finally for (b), suppose that $\sigma'(k)=\sigma'(k+1)$.
Since $\sigma'=\sigma'\sigma_j\delta_j=\sigma\delta_j$, 
we have $\sigma\delta_j(k)=\sigma'(k)=\sigma'(k+1)=\sigma\delta_j(k+1)$. Since $\delta_j(k)<\delta_j(k)+1\le \delta_j(k+1)$ and $\sigma$ is order-preserving, it
follows that $\sigma\delta_j(k)=\sigma(\delta_j(k)+1)$.
Since $\sigma$ is a shrink morphism,
$\sigma\ell_S\delta_j(k)=\sigma\delta_j(k)=\sigma'(k)$.
But $\delta_j=\sigma^*_j$, and so 
$\sigma\ell_S\delta_j(k)=\sigma'\sigma_j\ell_S\sigma^*_j(k)=\sigma'\ell_{S'}(k)$, whence 
$\sigma'\ell_{S'}(k)=\sigma'(k)$ as required.
\end{proof}

This allows us to describe a normal form for shrink morphisms. 

\begin{corollary}
  Any shrink morphism in \Fl can be written uniquely
as a composite 
$$\xymatrix{
(\ord m,u,S) \ar[r]^-{\sigma_{j_1}} & 
(\ord{m-1},u_1,S_1) \ar[r]^-{\sigma_{j_2}} & 
(\ord{m-2},u_2,S_2) \ar[r] & \ldots \ar[r]^-{\sigma_{j_r}} & (\ord n, v,T) }$$ 
% $\sigma_{j_r}\ldots\sigma_{j_2}\sigma_{j_1}$ 
with $j_1\le j_2\le\ldots\le j_s$.
\end{corollary}

We now describe how our rewriting for maps in 
$\Delta$ applies to shrink morphisms.

\begin{proposition}\label{prop:CS-rewrite}
  Suppose that $\sigma_{j+1}\colon(\ord m,u,S)\to(\ord{m-1},v,T)$ and $\sigma_i\colon(\ord{m-1},v,T)\to(\ord{m-2},w,U)$ are shrink morphisms with \hbox{$i\le j$.} Then there are shrink morphisms
$\sigma_i\colon(\ord m,u,S)\to(\ord{m-1},v',T')$ and $\sigma_j\colon(\ord{m-1},v',T')\to(\ord{m-2},w,U)$ for a
unique choice of $v'$ and $T'$, and $\sigma_i\sigma_{j+1}=\sigma_j\sigma_i$.
\end{proposition}

\begin{proof}
The composite $\sigma=\sigma_i\sigma_{j+1}$ is a shrink morphism, and $i$ is the least element of $\ord m$ with $\sigma(i)=\sigma(i+1)$, so by Proposition~\ref{prop:CS-factorization} there is a unique factorization of $\sigma$ as a shrink morphism $\sigma_i\colon(\ord m,u,S)\to(\ord{m-1},v',T')$ followed by a shrink morphism $\sigma'\colon(\ord{m-1},v',T')\to(\ord{m-2},w,U)$. But $\sigma'$ could only be $\sigma_j$ by surjectivity of $\sigma_i$. 
\end{proof}

The direction of the rewrite is significant here.
If $S\in\Tam_3$ consists of the triangles $(012)$ and $(023)$, corresponding to the bracketing $(II)I$, then 
there are shrink morphisms as in the solid
part of the diagram 
$$\xymatrix{
(\ord 3,\emptyset,S) \ar[r]^{\sigma_0} \ar@{.>}[d]_{\sigma_1} & 
(\ord 2,\emptyset,S') \ar[d]^{\sigma_0} \\
(\ord 2,\emptyset,T) \ar[r]_{\sigma_0} & (\ord 1,\emptyset,S'') }$$
but no shrink morphism 
$\sigma_1\colon(\ord 3,\emptyset,S)\to(\ord 2,\emptyset,T)$. 
(There is only one $T\in\Tam_2$.)

\subsection{Presentation of the category of shrink morphisms.}

It now follows that the category \Fl of shrink morphisms is generated by shrink morphisms of the form $\sigma_i\colon(\ord n,u,S)\to(\ord{n-1},v,T)$. Abstractly, these can be specified by giving an object $(\ord n,u,S)$ and an $i$ in the range $0\le i\le n-2$, satisfying the conditions $i\notin u$ and $\ell_S(i)=i$.
Then $v$ is given by $\sigma_i u$ and $T$ by 
$\ell_T = \sigma_i \ell_S \delta_{i}$. 

The only relations we need are those given in 
Proposition~\ref{prop:CS-rewrite}: $\sigma_i\sigma_{j+1}=\sigma_j\sigma_i$ whenever $i\le j$. We only need 
apply these in the forward direction, and we can apply
them for any such composable pair $\sigma_i\sigma_{j+1}$ of shrink morphisms. There will be such a composable pair with domain $(\ord n,u,S)$ whenever $0\le i\le j\le n-3$ with  $i,j+1\notin u$ and $\ell_S(i)=i$, and finally with either $\ell_S(j+1)=j+1$ or $\ell_S(j+1)=j=i$.

For future reference we record this as:

\begin{proposition}\label{prop:CS-presentation}
  The category \Fl of shrink morphisms may be 
presented as a category via the generators and relations
described above.
\end{proposition}

\subsection{Tensor product and shrink morphisms}

The objects of \Fl are the objects of \Fpm, and so we have a tensor product for them. There is an evident faithful functor $\Ul\colon \Fl\to\Delta_\bot$ which sends the tensor product on objects to the corresponding ordinal sum. 

\begin{proposition}
If $\sigma\colon(\ord m,u,S)\to(\ord{m'},u',S')$ and $\tau\colon(\ord n,v,T)\to(\ord{n'},v',T')$ are shrink morphisms then so is $\sigma+\tau\colon(\ord m,u,S)\ox(\ord n,v,T)\to(\ord{m'},u',S')\ox(\ord{n'},v',T')$.
\end{proposition}

\begin{proof}
It suffices to consider the case where one of the shrink morphisms is a generating surjection and the other is an identity. 

Suppose first that $\sigma$ is $\sigma_j\colon(\ord m,u,S)\to(\ord{m-1},u',S')$, while $\tau$ is the identity on $(\ord n,v,T)$. 
Note that $0\le j\le m-2$.
Then $\sigma+\tau$ is the function $\sigma_j\colon\ord{m+n}\to \ord{m-1+n}$. Since $\sigma$ is a shrink morphism, we have $j\notin u$, and so $j\notin u\star v$; again, since $\sigma$ is a shrink morphism we have $\ell_S(j)=j<m-1$, and so $\ell_{S\star T}(j)=\ell_S(j)$. This proves that $\sigma_j$ defines a shrink morphism with domain $(\ord m,u,S)\ox(\ord n,v,T)$. 

The codomain has left bracketing function 
$\sigma_j \ell_{S\star T} \delta_j$. Now
\begin{align*}
  \sigma_j \ell_{S\star T} \delta_j (k) &=
  \begin{cases}
    \sigma_j \ell_{S\star T} (k) & \text{if $k<j$} \\
    \sigma_j \ell_{S\star T}(k+1) & \text{if $k\ge j$} 
  \end{cases}
\\
&=
\begin{cases}
  \sigma_j \ell_S(k) & \text{if $k<j$} \\
  \sigma_j \ell_S(k+1) & \text{if $j\le k<m-2$} \\
  \sigma_j(0) & \text{if $k=m-2$} \\
  \sigma_j(m+\ell_T(k+1-m) & \text{if $k>m-2$}
\end{cases}
\\
&=
\begin{cases}
  \ell_{S'}(k) & \text{if $k<m-2$} \\
  0 & \text{if $k=m-2$} \\
  m-1+\ell_T(k+1-m) & \text{if $k>m-2$} 
\end{cases}
\\
&= \ell_{S'\star T}(k)
\end{align*}
and so the codomain does indeed have triangulation $S'\star T$.

We leave to the reader the case where $\sigma$ is the identity on $(\ord m,u,S)$ and $\tau$ is $\sigma_j\colon(\ord n,v,T)\to(\ord{n-1},v',T')$.
\end{proof}

% Now suppose that $\sigma$ is the identity on $(\ord m,u,S)$, while $\tau$ is $\sigma_j\colon(\ord n,v,T)\to(\ord{n-1},v',T')$. Then $\sigma+\tau$ is $\sigma_{m+j}\colon \ord{m+n}\to \ord{m+n-1}$. Since $\tau$ is a shrink morphism we have $j\notin v$, and so $m+j\notin u\cup v$; again, since $\tau$ is a shrink morphism we have $\ell_T(j)=j$, and so $\ell_{S\star T}(m+j)=m+\ell_T(j)=m+j$. This proves that $\sigma_{m+j}$ does define a shrink morphism with domain $(\ord m,u,S)\ox(\ord n,v,T)$. 

% The codomain has left bracketing function $\sigma_{m+j}\ell_{S\star T}\delta_{m+j}$. Now
% \begin{align*}
%   \sigma_{m+j}\ell_{S\star T}\delta_{m+j}(k) &=
%   \begin{cases}
%     \ell_S(k) & \text{if $k<m-1$} \\
%     0 & \text{if $k=m-1$} \\
%     \sigma_{m+j}\ell_{S\star T}(m+\delta_j(k-m)) & \text{if $k>m-1$} 
%   \end{cases}
% \\
% &=
% \begin{cases}
%     \ell_S(k) & \text{if $k<m-1$} \\
%     0 & \text{if $k=m-1$} \\
%     \sigma_{m+j}(m+\ell_{T}(\delta_{j}(k-m))) & \text{if $k>m-1$} 
% \end{cases}
% \\
% &=
% \begin{cases}
%     \ell_S(k) & \text{if $k<m-1$} \\
%     0 & \text{if $k=m-1$} \\
%     m+\sigma_j\ell_{T}\delta_j(k-m) & \text{if $k>m-1$} 
% \end{cases} \\
% &= \ell_{S\star T'}(k)
% \end{align*}
% and so the codomain does indeed have triangulation $S\star T'$.

It now follows, by faithfulness of the forgetful functor $\Ul\colon\Fl\to\dbot$, that the tensor product on $\Delta_\bot$ lifts to a tensor product on \Fl, strictly preserved by $U_\lambda$.

\begin{proposition}
  For any object $(\ord m,u,S)$ of \Fl, the function $\sigma_0\colon\ord{m+1}\to \ord m$ defines a morphism $(1,\emptyset,*)\ox(\ord m,u,S)\to(\ord m,u,S)$.
\end{proposition}

\begin{proof}
By definition of $(1,\emptyset,*)\ox(\ord m,u,S)$ the induced triangulation $T$ has $\ell_T(0)=0$. Also 
$\sigma_0\ell_T\delta_0(k)=\sigma_0\ell_T(k+1)=\sigma_0(1+\ell_S(k))=\ell_S(k)$.
\end{proof}

Using faithfulness of $\Ul$ once again, we deduce that these $\sigma_0$  are the components of a natural transformation $\lambda\colon I\ox-\to 1$. Thus \Fl is a skew left unital magmoidal category, and $\Ul$ preserves all of this structure. 

\subsection{Universal property of shrink morphisms.}

\begin{proposition}\label{prop:CS-univ-prop}
\Fl is the free skew left unital magmoidal category on 
one object.% is the category of constrained surjections.
\end{proposition}

\begin{proof}
Let \C be a skew left unital magmoidal category, 
and $C$ an object of \C. We know that there is 
a unique functor $F\colon\Fpm\to\C$ which strictly
preserves the tensor product of objects, and the unit
object, and which sends the generator 
$X=(1,1,*)$ to $C$. We need to show that it can be made
functorial with respect to shrink morphisms in a
 unique way such that the skew left unital structure
is preserved.

Consider a generating shrink morphism 
$\sigma_j\colon (\ord{m+1},u,S)\to(\ord m,u',S')$. We know that
$\ell_S(j)=j$.

If $j=0$ and $\ell_S(k)>0$ for all $k>0$, then in fact $(\ord{m+1},u,S)=I\ox (\ord m,u',S')$ and $\sigma_j$ is the component at $(\ord m,u',S')$ of $\lambda$. Thus such a generating shrink morphism must be sent to $\lambda\colon I\ox F(\ord m,u',S')\to F(\ord m,u',S')$ in \C.

Suppose otherwise, and consider the interval $[j,r_S(j+1)]$ given by $\{i\mid j\le r_S(j+1)\}$. We shall define an lbf $\ell_T$ on $[j,r_S(j+1)]$ by 
$$\ell_T(i) =
\begin{cases}
\ell_S(i) & \text{if $j\le i<r_S(j+1)$} \\
r_S(j) & \text{if $i=r_S(j+1)$} 
\end{cases}$$
It is clear that $\ell_T(i)\le i\le r_S(j+1)$ for all $i$ in the interval, and that
$\ell_T$ preserves the top element $r_S(j+1)$. We have $\ell_T(j)=\ell_S(j)=j$, while if $j+1\le i<r_S(j+1)$ then $j+1\le \ell_S(i)$ by the formula for $\ell_S$ in Proposition~\ref{bijbf}.  Thus $\ell_S(i)\le j$ for all $i$ with equality only for $i=j$. We must show that 
$\ell_T(i)\le h\le i$ implies $\ell_T(i)\le\ell_T(h)$. The cases $i=j$ and $i=r_S(j+1)$ are trivial, so suppose that $j<i<r_S(j+1)$. Then $\ell_S(i)\le h\le i$ and so $\ell_S(i)\le \ell_S(h)$ and so in turn $\ell_T(i)\le \ell_T(h)$. Thus we do have an lbf $\ell_T$ on $[j,k]$. (This could also be deduced from Proposition~\ref{prop:cob-for-lbf} using a suitable choice of $\xi$.)

Consider the ``fat'' object $([j,r_S(j+1)],v,T)$, where  $v$ is given by $[j,r_S(j+1)]\cap u$. Now $\sigma_j$ restricts to an order-preserving surjection $[j,r_S(j+1)]\to[j,r_S(j+1)-1]$ which defines a shrink morphism $\sigma\colon([j,r_S(j+1)],v,T)\to([j,r_S(j+1)-1],v',T')$, and now $j$ is the bottom element of $[j,r_S(j+1)]$ and $\ell_T(i)=j$ only if $i=j$. Up to an isomorphism re-indexing the elements of $[j,r_S(j+1)]$, then, our shrink morphism $\sigma\colon([j,r_S(j+1)],v,T)\to([j,r_S(j+1)-1],v',T')$ is the component at $([j,r_S(j+1)-1],v',T')$ of $\lambda$, and so should be sent to the corresponding component of the $\lambda$ in \C.

Furthermore, the original generating shrink morphism $\sigma_j\colon(\ord{m+1},u,S)\to(\ord m,u',S')$ is, up to re-indexing, obtained from $\sigma_j\colon([j,r_S(j+1)],v,T)\to([j,r_S(j+1)-1,v',T')$ by tensoring on either side with identity morphisms. We saw above how to define $F$ on $\sigma_j\colon([j,r_S(j+1)],v,T)\to([j,r_S(j+1)-1,v',T')$, and so tensoring the result in \C with suitable identity morphisms gives the required image under $F$ of the original $\sigma_j\colon(\ord{m+1},u,S)\to(\ord{m},u',S')$.

We have now defined $F$ on the generating morphisms;
this definition respects the $\lambda$'s and respects
iterated whiskering of the generating morphisms by 
identities. It remains to show that $F$ respects the 
relations; preservation of composition and of tensoring
will then follow.

Suppose then that we have a commutative square
$$\xymatrix{
(\ord m,u,S) \ar[r]^{\sigma_{j+1}} \ar[d]_{\sigma_i} &
(\ord{m-1},v,T) \ar[d]^{\sigma_i} \\
(\ord{m-1},v',T') \ar[r]_{\sigma_j} & (\ord{m-2},w,U) }$$
in \Fl with $i\le j$. We must show that it is mapped by $F$ to a commutative square in \C.

We know that $\ell_S(i)=i$ and $\ell_S(j+1)=j+1$. Thus there are $S$-triangles $(i,i+1,h)$ and $(j+1,j+2,k)$. If $h\le j+1$, then the square in \C will commute by functoriality of $\ox$.
If $h\not\le j+1$, then since $i+1\le j+1$ we must have $k\le h$. In this case the square in \C will commute by naturality of $\lambda$. \end{proof} 

\section{Combining associativity and left units}
\label{sect:la}

%\subsection{\F-surjections}

The shrink morphisms of the previous section are in particular surjective, but there are more general surjections which will be in \Fsk.

If $(M,u,S)$ and $(N,v,T)$ are objects of \Fpmfat, we define an {\em \F-surjection} from $(M,u,S)$ to $(M,v,T)$ to be an order-preserving surjection $\sigma\colon M\to N$ which defines a shrink morphism $(M,u,L)\to(N,v,T)$ for some $L\in\Tam_m$ with $S\le L$.

Every \F-surjection can be factorized as
$$\xymatrix{
(M,u,S) \ar[r]^{1_M} & (M,u,L) \ar[r]^{\sigma} & 
(N,v,T) }$$
but these factorizations need not be unique, since
there can be several possible choices of $L$. There is, however, a canonical choice, namely the maximal one, which exists by the following result.

\begin{proposition}
  Suppose that $\sigma\colon(M,u,S)\to(N,v,T)$ is an \F-surjection. Then the set of all $L_\gamma\in\Tam_M$ with $S\le L_\gamma$ and with $\sigma\colon(M,u,L_\gamma)\to(N,v,T)$ an \F-surjection has a greatest element.
\end{proposition}

\begin{proof} 
It will suffice to show that the join $L$ of the $L_\gamma$ lies in the set. Since joins of lbf's are constructed pointwise, we have $\ell_L(j) = \vee_\gamma \ell_{L_\gamma}(j)$. We shall write $\ell_\gamma$ for $\ell_{L_\gamma}$ and write $\ell$ for $\ell_L$.

First of all, if $\sigma(j)=\sigma(j+1)$ then for each $\gamma$ we have $\sigma\ell_\gamma(j)=\sigma(j)$. But now $\sigma\ell(j)=\sigma(\vee_\gamma \ell_\gamma(j))=\vee_\gamma \sigma\ell_\gamma(j)=\vee_\gamma\sigma(j)=\sigma(j)$, where we have used the fact that $\sigma$ preserves joins.

Secondly, $\sigma\ell\sigma^*(j)=\sigma(\vee_\gamma \ell_\gamma\sigma^*(j))=\vee_\gamma \sigma\ell_\gamma\sigma^*(j)=\vee_\gamma \ell_T(j)=\ell_T(j)$ and so $\sigma\ell\sigma^*=\ell_T$.

Since $u$ is unchanged for the different $\gamma$, this completes the proof. \end{proof}

Thus \F-surjections are generated by shrink morphisms and by morphisms of the form $(M,u,S)\to(M,u,S')$ with $S\le S'$; we shall call morphisms of the latter type {\em Tamari morphisms}.  We know that shrink morphisms can be composed, and that Tamari morphisms can be composed, but in order to have a category of \F-surjections we also need to know that the composite of an \F-surjection followed by a Tamari morphism is still an \F-surjection. We shall now turn to this.

First recall from Section~\ref{sect:orientals} that if $S$ contains triangles $(x_0x_1x_2)$ and $(x_0x_2x_3)$ then $x=(x_0x_1x_2x_3)$ defines an inequality $S\le T$, where $T$ is obtained from $S$ by replacing the two triangles $t_S(x_1)$ and $t_S(x_2)$ given above by $(x_0x_1x_3)$ and $(x_1x_2x_3)$. Furthermore, inequalities of this type generate the poset $\Tam$.

We can re-express this in terms of bracketings. To give triangles $(x_0x_1x_2)$ and $(x_0x_2x_3)$ as above is to give $c=x_1-1$ and $d=x_2-1$ with $\ell_S(c)=\ell_S(d)$, and with $r_S(c+1)=d$. Then the only difference between $\ell_S$ and $\ell_T$ is that $\ell_T(d)=c+1$. Similarly, the only difference between $r_S$ and $r_T$ is that $r_T(c+1)=r_S(d+1)$.
We will sometimes write $a_{c,d}$ for the inequality $S\le T$.

\begin{proposition}\label{prop:US-rewrite}
  If $\sigma\colon(M,u,S)\to(N,v,T)$ is a shrink morphism, and $T\le T'$ in $\Tam_n$, then there is
a shrink morphism $\sigma\colon(M,u,S')\to(N,v,T')$ for some $S'\in\Tam_m$ with $S\le S'$.
\end{proposition}

\begin{proof}
It will suffice to consider the case of a generating shrink morphism
$\sigma_j\colon(\ord m,u,S)\to(\ord{m-1},v,T)$ and a 
generating inequality $a_{c,d}\colon T\to T'$ in $\Tam_{m-1}$.

We know then that $\ell_S(j)=j$, and we also know that $c<d$, that $\ell_T(c)=\ell_T(d)$, and  that $d$ is minimal with the property that $\ell_T(d)\le c<d$. Furthermore, $\ell_{T'}(d)=c+1$ while $\ell_{T'}(k)=\ell_T(k)$ for all $k\neq d$.

We also know that $\ell_T=\sigma\ell_S\sigma^*=\sigma_j\ell_S\delta_j$. Thus $\sigma_j\ell_S\delta_j(c)=\sigma_j\ell_S\delta_j(d)$. Since $\delta_j$ is injective, we certainly have $\delta_j(c)<\delta_j(d)$.

\subsubsection*{Case 1: $\ell_S\delta_j(c)=\ell_S\delta_j(d)$.}%, and $\delta_j(d)=r_S(\delta_j(c)+1)$.}
%$\delta_j(d)$ is the least $h$ with $\ell_S(h)\le \delta_j(c)<h$.}

Since $c<d$ we have $\delta_j(c)<\delta_j(d)$, and now $\ell_S\delta_j(d)=\ell_S\delta_j(c)\le\delta_j(c)<\delta_j(d)$, and so $\ell_S\delta_j(d)<\delta_j(c)+1\le\delta_j(d)$. Thus $r_S(\delta_j(c)+1)\le \delta_j(d)$.

Suppose that the inequality is strict, so that $r_S(\delta_j(c)+1)<\delta_j(d)$. Then there exists an $h<\delta_j(d)$ with $\ell_S(h)<\delta_j(c)+1\le h$; also $\ell_S(h)\le\delta_j(c)$. Now $\ell_S(h)<h$ but $\ell_S(j)=j$, so $h\neq j$; but then $h=\delta_j(k)$ for some $k$, and $\ell_T(k)=\sigma_j\ell_S\delta_j(k)=\sigma_j\ell_S(h)\le\sigma_j\delta_j(c)=c<k$. Thus $\ell_T(k)<c+1\le k$, and so $r_{T}(c+1)\le k$; that is, $d\le k$. But then $\delta_j(d)\le\delta_j(k)=h$, contradicting the fact that $h<\delta_j(d)$.

Thus in fact $r_S(\delta_j(c)+1)=\delta_j(d)$. Then we have an inequality $a_{\delta_j(c),\delta_j(d)}\colon S\to S'$, where $\ell_{S'}$ agrees with $\ell_S$ except that $\ell_{S'}\delta_j(d)=\delta_j(c)+1$.

Let's show that $\sigma_j\colon (M,u,S')\to(N,v,T')$ is a shrink morphism. Since $u$ and $v$ are unchanged we need not worry about them. Since $j$ is certainly not $\delta_j(d)$, we have $\ell_{S'}(j)=\ell_S(j)=j$. Thus $\sigma_j$ does define a shrink morphism with domain $(M,u,S')$; it remains to show that the codomain has triangulation $T'$, or in other words that $\ell_{T'}=\sigma_j\ell_{S'}\delta_j$.

To do this, observe that 
$\sigma_j \ell_{S'} \delta_j (d) = \sigma_j(\delta_j(c)+1)=c+1=\ell_{T'}(d)$ while if $k\neq d$ then $\delta_j(k)\neq\delta_j(d)$ and  
$\sigma_j\ell_{S'}\delta_j(k)=\sigma_j\ell_S\delta_j(k)=\ell_T(k)=\ell'_T(k)$, and so $\sigma_j \ell_{S'}\delta_j=\ell'_T$ as required.

\subsubsection*{Case 2: $\ell_S\delta_j(c)\neq\ell_S\delta_j(d)$.} 
Since $\sigma_j\ell_S\delta_j(c)=\sigma_j\ell_S\delta_j(d)$ this can only occur if one of $\ell_S\delta_j(c)$ and $\ell_S\delta_j(d)$ is $j$ and the other is $j+1$. In any case, both $\ell_S\delta_j(c)$ and $\ell_S\delta_j(d)$ are greater than or equal to $j$, hence so too are $\delta_j(c)$ and $\delta_j(d)$, hence so too are $c$ and $d$. Now $c\ge j$, so that $\delta_j(c)\ge j+1$. Thus 
$\ell_S\delta_j(d)\le j+1\le\delta_j(c)\le\delta_j(d)$, and  so $\ell_S\delta_j(d)\le\ell_S\delta_j(c)$, so the only possibility is that $\ell_S\delta_j(d)=j$ and $\ell_S\delta_j(c)=j+1$.

We therefore know that $j\le c<d$. Now $d=r_T(c+1)=\sigma_j r_S \delta_{j+1}(c+1)=\sigma_j r_S(c+2)=r_S(c+2)-1$, where the last step uses the fact that $r_S(c+2)\ge c+2>j$. Thus $r_S(c+2)=d+1$.  Furthermore $\ell_Sr_S(c+2)=\ell_S(d+1)=j$, while $\ell_S(c+1)=j+1$, thus $\ell_Sr_S(c+2)\neq \ell_S(c+1)$; it follows by Proposition~\ref{fixbf}(v) that $r_S(c+2)=r_S\ell_S(c+1)=r_S(j+1)$ and so that $r_S(j+1)=d+1$.

At this point we have the following information.
\begin{itemize}
\item $j\le c<d$
\item $\ell_S(d+1)=\ell_S(j)=j$
\item $\ell_S(c+1)=j+1$
\item $r_S(j+1)=r_S(c+2)=d+1$.
\end{itemize}
Thus there is a generating inequality $a_{j,d+1}\colon S\to S''$, where $\ell_{S''}$ agrees with $\ell_S$ except that $\ell_{S''}(d+1)=j+1$.

Now $c+1<d+1$ and $\ell_{S''}(c+1)=\ell_S(c+1)=j+1=\ell_{S''}(d+1)$. Furthermore $r_{S''}(c+2)=r_{S}(c+2)=d+1$, and so we have a generating inequality $a_{c+1,d+1}\colon S''\to S'$. Here $\ell_{S'}$ agrees with $\ell_{S''}$, and so with $\ell_S$, except that $\ell_{S'}(d+1)=c+2$.

Clearly $S\le S''\le S'$, and we have an object $(M,u,S'')$ in \Flfat. Since $u$ is unchanged, and $\ell_{S''}(j)=\ell_S(j)=j$, there is a shrink morphism $\sigma_j$ with domain $(M,u,S'')$. It remains to show that its codomain has triangulation $T'$; in other words that $\sigma_j\ell_{S''}\delta_j = \ell_{T'}$.

If $k\neq d$, then $\delta_j(k)\neq d+1$, and now
$\sigma_j\ell_{S''}\delta_j(k)=\sigma_j\ell_S\delta_j(k)=\ell_T(k)=\ell_{T'}(k)$. On the other hand,
$\sigma_j\ell_{S''}\delta_j(d)=\sigma_j\ell_{S''}(d+1)=\sigma_j(c+2)=c+1=\ell_{T'}(d)$ as required.
\end{proof}

It follows from the proposition that a composite of \F-surjections is an \F-surjection, and so that we have a category \Flafat of \F-surjections. We write \Fla for the full subcategory consisting of objects $(M,u,S)$ for which $M$ is an ordinal $\ord m$, and write $\Ula$ for the evident faithful functor $\Fla\to\dbot$.

We shall now describe a presentation for the category, before turning to its universal property. First we need the following result.

\begin{proposition}\label{prop:surjection-reduction}
Suppose that $\sigma_j$ is a generating shrink morphism
$(\ord m,u,S)\to(\ord{m-1},v,T)$ and also $(\ord m,u',S')\to(\ord{m-1},v,T)$. Then $u=u'$ and 
$S$ and $S'$ are comparable in $\Tam_{m}$.
\end{proposition}

\begin{proof}
The fact that $u=u'$ is immediate from the fact that
$\sigma_j$ induces a bijection with $v$  of each of $u$
and $u'$.

% We shall use the fact that $S\le S'$ iff 
% $\ell_S(k)\le\ell_{S'}(k)$ for all $k$. 

The fact that we have shrink morphisms means
that $\ell_S(j)=j=\ell_{S'}(j)$.

If $k<j$ then $\sigma_j \ell_S(k)=\sigma_j \ell_S \delta_j(k)=\ell_T(k)=\sigma_j \ell_{S'}\delta_j(k)=\sigma_j \ell_{S'}(k)$; while $\sigma_j\ell_S(k)\le\ell_S(k)\le k<j$ and so in fact $\ell_S(k)=\ell_{S'}(k)$.

If $k>j$ then $\sigma_j \ell_S(k) =\sigma_j \ell_S\delta_j(k-1)=\ell_T(k-1)=\sigma_j\ell_{S'}\delta_j(k-1)=\sigma_j\ell_{S'}(k)$. It follows that $\ell_S(k)$ and $\ell_{S'}(k)$ can differ by at most one, and only if $\ell_T(k-1)=j$.

If $S\neq S'$ then there is a least $k$ with $\ell_S(k)\neq\ell_{S'}(k)$; without loss of generality we may suppose that $\ell_S(k)<\ell_{S'}(k)$. By the previous paragraph, we know that $\ell_S(k)=j$ and $\ell_{S'}(k)=j+1$. For any $h\neq k$ with $\ell_S(h)\neq\ell_{S'}(h)$ we must have $h>k$ and once again one of $\ell_S(h)$ and $\ell_{S'}(h)$ is $j$ and the other is $j+1$. But now $\ell_S(h)\le j+1\le k\le h$ and so $\ell_S(h)\le\ell_S(k)=j$; thus $\ell_S(h)=j$ and $\ell_{S'}(h)=j+1$. This proves that for all $h$ with $\ell_S(h)\neq\ell_{S'}(h)$ we have $\ell_S(h)<\ell_{S'}(h)$, and so $S\le S'$. 
\end{proof}

As well as the previous rewrite rules for shrink morphisms, we now add two further rewrite rules. 

First, in the context of Proposition~\ref{prop:US-rewrite}, we have the rule $a_{c,d}\sigma_j\to \sigma_j a_{\delta_j(c),\delta_j(d)}$. 
When does this actually apply? Write $c'=\delta_j(c)$ and $d'=\delta_j(d)$. We need $c'<d'$, $\ell_S(c')=\ell_S(d')$, and $r_S(c'+1)=d'$ in order to ensure that $a_{\delta_j(c),\delta_j(d)}$ exists, and of course we also need $\ell_S(j)=j$. It then follows that $\ell_T(c)=\sigma_j\ell_S\delta_j(c)=\sigma_j\ell_S(c')=\sigma_j\ell_S(d')=\sigma_j\ell_S\delta_j(d)=\ell_T(d)$. By Proposition~\ref{prop:cob-lr} we have $r_T(c+1)=\sigma_j r_S\delta_{j+1}(c+1)=\sigma_j r_S(\delta_j(c)+1)=\sigma_j r_S(c'+1)=\sigma_j(d')=d$. We also need $c<d$, or in other words $\sigma_j(c')<\sigma_j(d')$; of course $\sigma_j(c')\le \sigma_j(d')$ follows from $c'<d'$, so we only need to check that $\sigma_j(c')\neq \sigma_j(d')$. 
(The only way that we could have $c'<d'$ but $\sigma_j(c')=\sigma_j(d')$ is if $c'=j$ and $d'=j+1$.)

We have now expressed everything in terms of $c'$ and $d'$, so we may as well write $c$ and $d$ for these (then the old $c$ and $d$ will be given by $\sigma_j c$ and $\sigma_j d$.) Thus whenever $\ell_S(j)=j$, $\ell_S(c)=\ell_S(d)$, $r_S(c+1)=d$, but not both $c=j$ and $d=j+1$, we have a diagram 
\begin{equation}\label{eq:US-relation1}
\xymatrix{
(\ord m,u,S) \ar[r]^-{\sigma_j} \ar[d]_{a_{c,d}} & 
(\ord{m-1},v,T) \ar[d]^{a_{\sigma_j c,\sigma_j d}}  \\
(\ord m,u,S') \ar[r]_-{\sigma_j} & (\ord{m-1},v,T') }
\end{equation}
of \F-surjections, and we introduce the rewrite rule 
$$a_{\sigma_j c,\sigma_j d}\,\sigma_j\to \sigma_j \,a_{c,d}.$$

The second rule involves Proposition~\ref{prop:surjection-reduction}, and in particular the situation where $a_{c,d}\colon S\le S'$ is a generating inequality in $\Tam_m$, and $\sigma_j$ is 
a shrink morphism both from $(\ord m,u,S)\to(\ord{m-1},v,T)$ and from $(\ord m,u,S')\to(\ord{m-1},v,T)$. When will this arise? We need $c<d$ with $\ell_S(j)=j=\ell_{S'}(j)$, $\ell_S(c)=\ell_S(d)$, and $r_S(c+1)=d$. Since $\ell_S(d)\neq\ell_{S'}(d)$ but $\ell_S(j)=\ell_{S'}(j)$ we cannot have $d=j$. Thus $d=\delta_j(k)$ for some (unique) $k$. We still need to encode the condition that $\sigma_j\ell_S\delta_j=\sigma_j\ell_{S'}\delta_j$. This will certainly be true for all values $h$ with $\delta_j(h)\neq d$; in other words for $h\neq k$. So the condition reduces to $\sigma_j\ell_S\delta_j(k)=\sigma_j\ell_{S'}\delta_j(k)$, or in other words $\sigma_j\ell_S(d)=\sigma_j\ell_{S'}(d)$. Now $\ell_{S'}(d)=c+1$, while $\ell_S(d)=\ell_S(c)\le c<c+1$, so the only possibility is that $\ell_S(d)=j$ and $c+1=j+1$; in other words, $\ell_S(d)=c=j$.

Summarizing, if $j<d$, $\ell_S(d)=\ell_S(j)=j$ and $r_S(j+1)=d$, then we have a diagram of \F-surjections
\begin{equation}\label{eq:US-relation2}
\xymatrix{
(\ord m,u,S) \ar[dr]^{\sigma_j} \ar[d]_{a_{j,d}} \\
(\ord m,u,S') \ar[r]_-{\sigma_j} & (\ord{m-1},v,T) }
\end{equation}
and we impose the rewrite rule $\sigma_j\to\sigma_j a_{j,d}$.

One might expect to see another rule, covering Case~2 of Proposition~\ref{prop:US-rewrite}, but this turns out not to be necessary. In Case~2, we have \F-surjections as in the solid part of the diagram
$$\xymatrix{
(\ord m,u,S) \ar[d]_{a_{j,d+1}} \ar[r]^-{\sigma_j} & (\ord{m-1},v,T) \ar[dd]^{a_{c,d}} \\
(\ord m,u,S'') \ar[d]_{a_{c+1,d+1}} \ar@{.>}[ur]_{\sigma_j} \\
(\ord m,u,S') \ar[r]_-{\sigma_j} & (\ord{m-1},v,T') }$$
but we can decompose it, using the dotted line, into one instance of each rule. 

\begin{proposition}
The category \Fla of \F-surjections is generated by the generating shrink morphisms along with the generating Tamari morphisms $a_{c,d}\colon (\ord m,u,S)\to(\ord m,u,S')$.  It can be presented by these generators along with the relations for shrink morphisms, the relations that hold in the Tamari lattices $\Tam_m$, and the relations in \eqref{eq:US-relation1} and \eqref{eq:US-relation2}.  
\end{proposition}

We shall use the name {\em left unital skew monoidal category} for a skew semimonoidal category $(\C,\ox,\alpha)$
equipped with an object $I$, and a natural 
transformation $\lambda$ from $I\ox-$ to the identity endofunctor of \C, 
satisfying the coherence condition
$$\xymatrix{
(I\ox A)\ox B \ar[r]^{\alpha} \ar[dr]_{\lambda\ox1} & 
I\ox(A\ox B) \ar[d]^{\lambda} \\
& A\ox B. }$$

Since the tensor product of shrink morphisms is a shrink morphism, and the tensor product of 
morphisms in $\Tam_n$ is a morphism in $\Tam_n$, it follows that the tensor product of \F-surjections is an \F-surjection, and so that the tensor product functor on $\Delta_\bot$ lifts strictly through \Ula to \Fla.
% the category of
% unconstrained surjections.

We have maps $\lambda$ coming from the shrink morphisms, and $\alpha$ coming from the Tamari
posets; the fact that they are each natural with 
respect to all \F-surjections and that they
are mutually compatible follows from the faithfulness of \Ula along with the corresponding facts about $\Delta_\bot$. Thus \Fla is a left unital skew monoidal category.

\begin{theorem}\label{thm:US-univ-prop}
%  The category of unconstrained surjections is 
\Fla is the free left unital skew monoidal category 
generated by a single object. 
\end{theorem}

\begin{proof}
Let \C be a left unital skew monoidal category,
and $C$ an object of \C. We must show that there
is a unique structure-preserving functor $F$ from \Fla 
to \C sending the generator $X$ to \C. 

The uniqueness part is immediate: by the universal
property of $\Tam$ we know where all  Tamari morphisms must go, and by the universal property of \Fla we know where all shrink morphisms must go. 

The only thing to do is to check that when we define
$F$ in this way the relations \eqref{eq:US-relation1}
and \eqref{eq:US-relation2} are respected.

Consider first \eqref{eq:US-relation1}. We have triangles $(b',c',d')$ and $(b',d',e')$ and $(j,j+1,q)$ in $S$.
% , with 
% the equations $c'=j$ and $d'=j+1$ not both true. 
Then we are in one of the following cases:
\begin{enumerate}[(a)]
\item $q\le b'$
\item $j+1\le b'<e'\le q$
\item $b'\le j<q\le c'$
\item $c'\le j<q\le d'$
\item $d'\le j<q\le e'$
\item $e'\le j$.
\end{enumerate}
In each case the relation will hold in \C by 
functoriality of $\ox$ and/or naturality of $\lambda$.

Now consider \eqref{eq:US-relation2}, which involves  a triangle $(j,j+1,d)$. The relation will hold in \C because of the axiom
$$\xymatrix{
(I\ox A)\ox B \ar[r]^{\alpha} \ar[dr]_{\lambda\ox1} & 
I\ox(A\ox B) \ar[d]^{\lambda} \\
& A\ox B. }$$
\end{proof}

\section{Right skew units and duality}
\label{sect:r}

%\subsection{right skew units}

Given a pointed magmoidal category \C, write
$\C\oprev$ for the pointed magmoidal category
with underlying category $\C\op$ and with the 
reverse tensor product. 

To give $\C\oprev$ a skew semimonoidal
structure is equivalent to giving \C is skew
semimonoidal structure. A left unit for \C is 
equivalent to a right unit for the corresponding
skew semimonoidal category $\C\oprev$.
Thus the free right unital skew semimonoidal
category \Far on one object should be $\Fla\oprev$.

%  the reverse
% of the opposite of the free left unital skew 
% semimonoidal category on one object.

But we shall instead give a different model, which will allow us to see more directly the forgetful functor $\Far\to\dbot$. It is useful at first to work with \Flafat rather than \Fla.

For each object $(M,u,S)$ of \Flafat there is an associated object $(M\op,u,S\op)$, where $\ell_{S\op}=r_S$; we call this object $(M,u,S)\op$. This process defines an involution on the set of objects of \Flafat. Furthermore, we have $(M,u,S)\op\ox(N,v,T)\op=\left(\left(N,v,T\right)\ox\left(M,u,S\right)\right)\op$; this includes in particular the fact that  the ordinal sum $M\op+N\op$ is given by $(N+M)\op$.

If $\sigma\colon M\to N$ is an order-preserving surjection, then it has a right adjoint $\sigma^*\colon N\to M$ which is an order-preserving injection. The adjunction $\sigma\dashv \sigma^*$ between $M$ and $N$ can also be seen as an adjunction $\sigma^*\dashv\sigma$ between $N\op$ and $M\op$. 

The category of order-preserving surjections is contravariantly isomorphic to the category of order-preserving and left adjoint injections; the isomorphism sends $M$ to $M\op$ and $\sigma\colon M\to N$ to $\sigma^*\colon N\op\to M\op$.

Our ``fat'' version of the free right unital skew semimonoidal category on one object will be the category with the same objects as \Flafat, in which a morphism from $(N,v,T)$ to $(M,u,S)$ is an injective left-adjoint $\delta\colon N\to M$, for which the right adjoint $\delta^*\colon M\to N$ defines an \F-surjection $(M\op,u,S\op)\to(N\op,v,T\op)$. We call such a $\delta$ an {\em \F-injection}. When $\delta^*$ is not just an \F-surjection but a shrink morphism, we call $\delta$ a {\em swell morphism}.

The assignments $(M,u,S)\mapsto(M,u,S)\op$ and $\sigma\mapsto\sigma^*$ define an isomorphism from $(\Flafat)\oprev$ to the category \Farfat of \F-injections. It is useful, however to have a more explicit description of the swell morphisms and \F-injections. 

The first condition for $\sigma$ to be a shrink morphism is that $\sigma^*$ induce a bijection between $v$ and $u$. But $\sigma^*$ corresponds to what we are calling $\delta$, so the condition amounts to the requirement that $\delta$ induce a bijection between $v$ and $u$. 

The second condition for $\delta^*\colon M\op\to N\op$ to be a shrink morphism  involves the successor of an element $j$. But the successor in $M\op$ is the predecessor in $M$. So the condition becomes: if $\delta^*(j)=\delta^*(j-1)$ then $\delta^*\ell_{S\op}(j)=\delta^*(j)$; or, equivalently, if $\delta^*(j)=\delta^*(j-1)$ then $\delta^*r_S(j)=\delta^*(j)$.

To say that $\delta^*(j)=\delta^*(j-1)$ is equivalent to saying that $j\notin\im(\delta)$. To say that $\delta^*r_S(j)=\delta^*(j)$ is to say that if $j<i\le r_S(j)$ then $i\notin\im(\delta)$. Thus the condition states that if $j\notin\im(\delta)$ and $j<i\le r_S(j)$ then $i\notin\im(\delta)$. Taking this contrapositive gives: if $j<i\le r_S(j)$ and $i\in\im(\delta)$ then $j\in\im(\delta)$. Finally if $i\in\im(\delta)$ then we may write $i=\delta(k)$, and so obtain the condition: if $j<\delta(k)\le r_S(j)$ then $j\in\im(\delta)$.

The third condition for $\delta^*\colon M\op\to N\op$ to be a shrink morphism involves a right adjoint to this $\delta^*$. This is the same as a left adjoint to $\delta^*\colon M\to N$, and this is just $\delta$. Thus the second condition becomes $\ell_{T\op}=\delta^*\ell_{S\op}\delta$, or equivalently $r_T=\delta^* r_S \delta$.

We summarize this as follows.
A {\em swell morphism} from $(N,v,T)$ to $(M,u,S)$ is an injective left adjoint $\delta\colon N\to M$ satisfying the following conditions:
\begin{enumerate}[(a)]
\item $\delta$ induces a bijection between $v$ and $u$
\item If $j<\delta(k)\le r_S(j)$, then $j\in\im(\delta)$
\item $r_T=\delta^* r_S \delta$.
\end{enumerate}

\begin{proposition}
  If $(M,u,S)$ is an object of \Fpmfat, an injective left adjoint $\delta\colon N\to M$ determines a swell morphism with codomain $(M,u,S)$ if and only if 
  \begin{enumerate}[(a)]
  \item $u\subseteq\im(\delta)$
    \item if $j<\delta(k)\le r_S(j)$ then $j\in\im(\delta)$
  \end{enumerate}
The domain is then $(N,v,T)$, where $v=\delta^{-1}(u)$ and $r_T=\delta^* r_S \delta$; or equivalently $\ell_T=\delta^* \ell_S \delta_*$, where $\delta_*$ denotes the right adjoint of the order-preserving surjection $\delta^*\colon M\to N$. 
\end{proposition}

\begin{proof}
Since $\delta$ is injective, it will induce a bijection from $\delta^{-1}(u)$ to $u$ just when  $u\subseteq\im(\delta)$. Thus condition (a) in the proposition is equivalent to condition (a) in the definition; also condition (b) is unchanged. The fact that $r_T=\delta^* r_S\delta$ is equivalent to $\ell_T=\delta^*\ell_S\delta_*$ follow from Proposition~\ref{prop:cob-lr}.
\end{proof}

Once again, it is useful to analyze what this says in the case of a generating injection. Let $\delta_i\colon \ord n\to \ord{n+1}$ be the injective order-preserving map whose image does not contain $i$. Since we want $\delta_i$ to be a left adjoint, we ask that $i\neq0$. Then the right adjoint $\delta^*_i$ of $\delta_i$ is $\sigma_{i-1}$.  This in turn has a right adjoint, called $(\delta_i)_*$ in the language of the last proposition, and given by $\delta_{i-1}$.

\begin{proposition}
If $i\neq0$ then $\delta_i\colon\ord n\to\ord{n+1}$ defines a swell morphism with codomain  $(\ord{n+1},u,S)$ if and only if $i\notin u$ and $\ell_S(i)\neq i$. The domain is then $(\ord n,v,T)$, where $v=\delta^{-1}_i(u)$ and where $\ell_T=\sigma_{i-1}\ell_S\delta_{i-1}$ and $r_T=\sigma_{i-1}r_S\delta_i$.
\end{proposition}

\begin{proof} 
To say that $u\subseteq\im(\delta)$ is to say that $i\notin u$. On the other hand, $j\in\im(\delta)$ is true for all $j$ except $j=i$, so we just need to ensure that $i<\delta(k)\le r_S(i)$ is impossible; equivalently that $i<r_S(i)$ is false. But this in turn amounts, by Proposition~\ref{fixbf}(i), to saying that $\ell_S(i)<i$.
\end{proof}

Now an {\em \F-injection} from $(N,v,T)$ to $(M,u,S)$ is an injective left adjoint $\delta\colon N\to M$ which defines a swell morphism $(N,v,T)\to (M,u,S')$ for some $S'\le S$.
The \F-injections are the morphisms of a category \Farfat with the same objects as \Fpmfat.
In order to slim down the category \Farfat, we restrict to those objects $(M,u,S)$ for which $M$ is an ordinal $\ord m$, and call the resulting category \Far.

As a formal dual of Theorem~\ref{thm:US-univ-prop}
we have:

\begin{theorem}\label{thm:UI-univ-prop}
  The category \Far of \F-injections is 
the free right unital skew monoidal category 
generated by a single object. 
\end{theorem}

\section{Skew monoidal categories}
\label{sect:lar}

\subsection{General maps} 

An {\em \F-morphism} from 
$(\ord m,u,S)\to(\ord n,v,T)$ is an order-preserving 
function $\phi\colon \ord m\to \ord n$ which can be written as the composite 
$$\xymatrix{
(\ord m,u,S) \ar[r]^{\sigma} & (\ord p,w,R) \ar[r]^{\delta} & (\ord n,v,T) }$$
of an \F-surjection $\sigma$ and an \F-injection $\delta$.

\begin{remark}
  \begin{enumerate}[(a)]
  \item Any order-preserving $\phi\colon\ord m\to \ord n$ has a unique factorization as an order-preserving surjection $\sigma\colon\ord m\to\ord p$ followed by an order-preserving injection $\delta\colon\ord p\to\ord n$.
\item The fact that $\delta$ has a right adjoint is equivalent to the fact that it preserves the bottom element. Since $\sigma$ preserves the bottom element as well, this is equivalent to the composite $\phi$ preserving the bottom element, and so to the composite $\phi$ having a right adjoint. 
\item The subset $w\subseteq\ord p$ is necessarily the image $\sigma u$. A necessary and sufficient condition for $\sigma^*$ to induce a bijection $w\cong u$ and $\delta$ to induce a bijection $w\cong v$ is that $\phi$ and $\phi^*$ induce a mutually inverse pair of bijections $u\cong v$.
\item Unlike the other parts of the structure, the triangulation $R$ is not uniquely determined. Implicit in the proof of Proposition~\ref{prop:UG-presentation} below is the fact that there is a canonical choice for $R$.
  \end{enumerate}
\end{remark}

\subsection{Composition}

Once again, work is needed to prove that these \F-morphisms compose. It will suffice to show that we can compose maps of the form 
$$\xymatrix{
(\ord m,u,S) \ar[r]^-{\delta_i} & (\ord{m+1},v,U) \ar[r]^1 & 
(\ord{m+1},v,V) \ar[r]^-{\sigma_j} & (\ord m,w,T). }$$
We begin by simplifying this situation, using the opposites of the rewriting rules \eqref{eq:US-relation1} and \eqref{eq:US-relation2} and the corresponding rules for swell morphisms.

First consider all the $V'\in\Tam_{m+1}$ with 
$U\le V'\le V$ and with $\ell_{V'}(j)=j$. 
By Proposition~\ref{prop:wedge-ells} the  
collection of all $V'\in\Tam_{m+1}$ with 
$\ell_{V'}(j)=j$ is closed in $\Tam_{m+1}$ under 
meets, and so we can let $V'$ be minimal
with the property that $\ell_{V'}(j)=j$
and $U\le V'\le V$. 

Since $\ell_{V'}(j)=j$ by definition of $V'$, and $j\notin v$, there is a shrink morphism 
$\sigma_j\colon(\ord{m+1},v,V')\to(\ord m,w,T')$ where $T'$ is given by $\ell_{T'}=\sigma_j\ell_{V'}\delta_j$. Then $\ell_{T'}(k)=\sigma_j\ell_{V'}\delta_j(k)\le\sigma_j\ell_V\delta_j(k)=\ell_T(k)$, where the first equality holds by definition of $T'$, the inequality by the fact that $V'\le V$ and $\sigma_j$ is order-preserving, and the last equality holds since $\sigma_j\colon(\ord{m+1},v,V)\to(\ord m,w,T)$ is a shrink morphism. Thus $T'\le T$. Now we have 
$$\xymatrix{
(\ord m,u,S) \ar[r]^-{\delta_i} & (\ord{m+1},v,U) \ar[r]^{1} & 
(\ord{m+1},v,V') \ar@{.>}[r]^1 \ar[d]_{\sigma_j} &
(\ord{m+1},v,V) \ar@{.>}[d]^{\sigma_j} \\
 && (\ord m,w,T') \ar@{.>}[r]^1 & (\ord m,w,T) }$$
and it will clearly suffice to ignore $(\ord{m+1},v,V)$
and $(\ord m,w,T)$ and show that the solid
part of the diagram has a composite. 
Thus we may
as well suppose from the beginning that there 
is {\em no} $W$ with $U\le W<V$ and $\ell_W(j)=j$.

Similarly, we can replace $U$ and $S$ by $U'\ge U$ and $S'\ge S$ if necessary, to reduce to the case where there is no $W$ with $U<W\le V$ and $\ell_W(i)\neq i$.

\subsubsection*{Case 1: $U=V$.}

Then we have a swell morphism $\delta_i\colon(\ord m,u,S)\to(\ord{m+1},v,V)$ and a shrink morphism $\sigma_j\colon(\ord{m+1},v,V)\to(\ord m,w,T)$. This is possible only if $\ell_V(j)=j$ and $\ell_V(i)\neq i$, which rules out the case $i=j$.

\subsubsection*{Case 1a: $i<j$.}

As functions, we have a factorization $\sigma_j\delta_i=
\delta_i\sigma_{j-1}$. We shall show that this lifts to a factorization of shrink and swell morphisms.

Since $\sigma_j\colon(\ord{m+1},v,V)\to(\ord m,w,T)$ is a shrink morphism $j\notin v$; but $j=\delta_i(j-1)$, and so $j-1\notin u$. Also $\ell_S(j-1)=\sigma_{i-1}\ell_V\delta_{i-1}(j-1)=\sigma_{i-1}\ell_V(j)=\sigma_{i-1}(j)=j-1$. Thus $\sigma_{j-1}$ defines a shrink morphism $(\ord m,u,S)\to(\ord{m-1},u',S')$ where $u'=\sigma_{j-1} u$ and $\ell_{S'}=\sigma_{j-1}\ell_S\delta_{j-1}$.

Since $\delta_i\colon(\ord m,u,S)\to(\ord{m+1},v,V)$ is a swell morphism, $i\notin v$ and so $i=\sigma_j(i)\notin w$. Also 
$\ell_T(i)=\sigma_j\ell_V\delta_j(i)=\sigma_j\ell_V(i)$ and $i=\sigma_j(i)$, while $\ell_V(i)<i$, and so $\ell_T(i)=i$ could occur only if $\ell_V(i)=j$ and $i=j+1$, but $i=j+1$ is specifically excluded from this case. Thus $\delta_i$ defines a swell morphism $(\ord{m-1},w',T')\to(\ord m,w,T)$ where $w'=\delta^{-1}_iw$ and 
$\ell_{T'}=\sigma_{i-1}\ell_T\delta_{i-1}$.

We now show that $w'=u'$ and $T'=S'$. First of all, 
$\delta_i w'=w=\sigma_j v=\sigma_j\delta_i u=\delta_i\sigma_{j-1}u=\delta_i u'$, and so $w'=u'$.
Secondly,  $\ell_{S'}=\sigma_{j-1} \ell_S \delta_{j-1} = \sigma_{j-1} \sigma_{i-1} \ell_V \delta_{i-1} \delta_{j-1}=\sigma_{i-1}\sigma_j\ell_V\delta_j\delta_{i-1}=\sigma_{i-1}\ell_T\delta_{i-1}=\ell_{T'}$, and so $S'=T'$. Thus we have a commutative square 
\begin{equation}
  \label{eq:UGrelation1}
  \xymatrix{
(\ord m,u,S) \ar[r]^{\delta_i} \ar[d]_{\sigma_{j-1}} & 
(\ord{m+1},v,V) \ar[d]^{\sigma_j} \\
(\ord{m-1},u',S') \ar[r]_{\delta_i} & (\ord m,w,T) }
\end{equation}
in which the vertical maps are shrink morphisms and the horizontal ones are swell morphisms, and so the composite is an \F-morphism.

\subsubsection*{Case 1b: $i=j+1$.}

In this case, the function $\sigma_j\delta_i$ is the identity, so we need to show that $S\le T$. Now 
$\ell_S=\sigma_{i-1}\ell_V\delta_{i-1}=\sigma_j\ell_V\delta_j=\ell_T$ and so in fact $S=T$. Thus the triangle 
\begin{equation}
  \label{eq:UGrelation2}
  \xymatrix{
(\ord m,u,S) \ar[r]^{\delta_i} \ar[dr]_1 &
(\ord{m+1},v,V) \ar[d]^{\sigma_j} \\
& (\ord m,w,T) }
\end{equation}
commutes.

\subsubsection*{Case 1c: $i>j+1$.}

This time we have a factorization $\sigma_j\delta_i=\delta_{i-1}\sigma_j$ as functions. 

Since $\sigma_j\colon(\ord{m+1},v,V)\to(\ord m,w,T)$ 
is a shrink morphism, $j\notin v$; but $j=\delta_i(j)$ and so $j\notin u$. Also 
$\ell_S(j)=\sigma_{i-1}\ell_V\delta_{i-1}(j)=\sigma_{i-1}\ell_V(j)=\sigma_{i-1}(j)=j$. Thus $\sigma_j$ defines a shrink morphism $(\ord m,u,S)\to(\ord{m-1},u',S')$ where $u'=\sigma_j u$ and $\ell_{S'}=\sigma_j\ell_S\delta_j$.

Since $\delta_i\colon(\ord m,u,S)\to(\ord{m+1},v,V)$ is a swell morphism, $i\notin v$ and so $i-1=\sigma_j(i)\notin w$. Also $r_T(i-1)=\sigma_{j-1}r_V\delta_{j-1}(i-1)=\sigma_{j-1}r_V(i)=\sigma_{j-1}(i)=i-1$. Thus $\delta_{i-1}$ defines a swell morphism $(\ord{m-1},w',T')\to(\ord m,w,T)$ where $w'=\delta^{-1}_{i-1}w$ and $\ell_{T'}=\sigma_{i-2}\ell_T\delta_{i-2}$.

We now show that $u'=w'$ and $S'=T'$. First of all, $\delta_{i-1}w'=w=\sigma_j v=\sigma_j\delta_i u=\delta_{i-1}\sigma_j u=\delta_{i-1} u'$ and so $w'=u'$. Secondly, $\ell_{S'}=\sigma_j\ell_S\delta_j=\sigma_j\sigma_{i-1}\ell_V\delta_{i-1}\delta_j=\sigma_{i-2}\sigma_j\ell_V\delta_j\delta_{i-2}=\sigma_{i-2}\ell_T\delta_{i-2}=\ell_{T'}$, and so $S'=T'$ as required. Thus we have a commutative square 
\begin{equation}
  \label{eq:UGrelation3}
  \xymatrix{
(\ord m,u,S) \ar[r]^{\delta_i} \ar[d]_{\sigma_{j}} & 
(\ord{m+1},v,V) \ar[d]^{\sigma_j} \\
(\ord{m-1},u',S') \ar[r]_{\delta_{i-1}} & (\ord m,w,T) }
\end{equation}
in which the vertical maps are shrink morphisms and the horizontal ones are swell morphisms, and so the composite is an \F-morphism.

\subsubsection*{Case 2: $U<V$.}
By minimality of $V$, we must have $\ell_U(j)<j$, and so there are $U$-triangles $(\ell_U(j),j+1,q)$ and $(\ell_U(j),k+1,j+1)$ for some $k<j$; then $\ell_U(k)=\ell_U(j)$. We now have a generating inequality $a_{j,k}\colon U\le W$; explicitly $\ell_W$ agrees with $\ell_U$ except that $\ell_W(j)=k+1$ while $\ell_U(j)<k+1$.

Furthermore, $\ell_W(j)=k+1\le j=\ell_V(j)$, while if $h\neq j$ then $\ell_W(h)=\ell_U(h)\le \ell_V(h)$, so also $W\le V$.
Thus by maximality of $U$, we must have $\ell_W(i)=i$.
But $\ell_U(i)\neq i$, and $\ell_W$ and $\ell_U$ agree except at $j$, so we must have $i=j=k+1$. 
Now $\ell_W(j)=k+1=j$, and so by minimality of $V$ we must have $V=W$.

Since $\ell_V$ and $\ell_U$ now agree except at $j$, we have $\ell_V\delta_j=\ell_U\delta_j$, and so $\ell_T=\sigma_j\ell_V\delta_j=\sigma_j\ell_U\delta_j$. On the other hand $\ell_S=\sigma_{j-1}\ell_U\delta_{j-1}$, so $\ell_S$ and $\ell_T$ can only differ at $j-1$ or at some $h$ with $\ell_U\delta_{j-1}(h)=j$.

If $\ell_U\delta_{j-1}(h)=j$, then $j\le \delta_{j-1}(h)$ and so $j\le h+1$ and $j-1\le h$. Thus we have $\ell_U(h+1)=j$. But $\ell_U$ is idempotent and so $\ell_U(j)=j$, which contradicts $\ell_U(j)<j$. Thus $\ell_S$ and $\ell_T$ agree except possibly at $j-1$ (which is equal to $k$).

For this last case, we have on the one hand $\ell_U\delta_{j-1}(j-1)=\ell_U(j)=\ell_U(k)=\ell_U(j-1)$, and so $\ell_S(j-1)=\sigma_{j-1}\ell_U\delta_{j-1}(j-1)=\sigma_{j-1}\ell_U(j-1)=\ell_U(j-1)$, where the last step uses the fact that $\ell_U(j-1)\le j-1$.  On the other hand, $\ell_T(j-1)=\sigma_j\ell_U\delta_j(j-1)=\sigma_j\ell_U(j-1)=\ell_U(j-1)$, and so in fact $\ell_S(j-1)=\ell_T(j-1)$.

Thus $\ell_S(h)=\ell_T(h)$ for all $h$, and $S=T$. We therefore have a commutative diagram
\begin{equation}
  \label{eq:UGrelation4}
  \xymatrix{
(\ord{m+1},v,U) \ar[r]^{a_{j,j+1}} &
(\ord{m+1},v,V) \ar[d]^{\sigma_j} \\
(\ord m,u,S) \ar[r]_1 \ar[u]^{\delta_j}  & (\ord m,w,S) }
\end{equation}

Thus we can compose \F-maps and they
form a category \F, with an evident faithful functor $U\colon\F\to\dbot$.

\subsection{Presentation}

We now describe a presentation for \F. The generators will consist of the generating shrink morphisms, the generating swell morphisms, and the generating inequalities in $\Tam$.

The relations will consist of the previously given relations for \F-surjections, the dual relations for \F-injections, and the relations corresponding to 
\eqref{eq:UGrelation1}, \eqref{eq:UGrelation2},
\eqref{eq:UGrelation3},
and \eqref{eq:UGrelation4}.

\begin{proposition}
  \label{prop:UG-presentation}
The given generators and relations constitute a 
presentation for the category \F
%of \F-maps.
\end{proposition}

\begin{proof}
The \F-surjections are generated by the generating Tamari morphisms and the generating shrink morphisms, while the \F-injections are generated by the generating Tamari morphisms and the generating swell morphisms. Since a general \F-morphism is a composite of an \F-surjection and an \F-injection it follows that the \F-morphisms are generated by the generating Tamari morphisms, the generating shrink morphisms, and the generating swell morphisms.

The given relations all hold in \Fsk since they hold in \dbot and $U$ is faithful. Why do these relations suffice? They certainly allow us to write any composite of the generators as an \Fsk-surjection followed by an \Fsk-injection, and the surjection in injection are uniquely determined as functions. Suppose then that we have morphisms as in the exterior (solid) part of the diagram$$\xymatrix{
& (\ord m,u,S_1) \ar[r]^{\sigma} & 
(\ord p,w,R_1) \ar[r]^{\delta}  & 
(\ord n,v,T_1) \ar[dr]^{1_{\ord n}} \\
(\ord m,u,S) \ar[ur]^{1_{\ord m}} \ar[dr]_{1_{\ord m}} \ar@{.>}[r]^{1_{\ord m}} &
(\ord m,u,S_3) \ar@{.>}[r]^{\sigma} \ar@{.>}[u]^{1_{\ord m}} \ar@{.>}[d]_{1_{\ord m}} & 
(\ord p,w,R_3) \ar@{.>}[u]^{1_{\ord p}} \ar@{.>}[d]_{1_{\ord p}} \ar@{.>}[r]^{\delta} & 
(\ord n,v,T_3) \ar@{.>}[u]^{1_{\ord n}} \ar@{.>}[d]_{1_{\ord n}}&
(\ord n,v,T) \\
& (\ord m,u,S_2) \ar[r]_{\sigma} & 
(\ord p,w,R_2) \ar[r]_{\delta} & (\ord n,v,T_2) \ar[ur]_{1_{\ord n}} }$$
in which the maps labelled $\sigma$ are shrink morphisms, the maps labelled $\delta$ are swell morphisms, the maps labelled $1$ are Tamari morphisms, and the exterior commutes. 

Write $T_3$ for the meet $T_1\wedge T_2$. Meets of right bracketing functions are calculated pointwise, so $r_{T_3}(j)=r_{T_1}(j)\wedge r_{T_2}(j)$ for all $j$. If $j<\delta(k)\le r_{T_3}(j)$, then we have $j<\delta(k)\le r_{T_i}(j)$ for $i=1,2$, and so $j\in\im(\delta)$. Thus there is a swell morphism $\delta\colon(\ord p,w,R_3)\to(\ord n,v,T_3)$, where $R_3$ is given by $r_{R_3}=\delta^* r_{T_3} \delta$. Since $\delta^*$ preserves meets, it follows that $r_{R_3}=r_{R_1}\wedge r_{R_2}$ and so that $R_3=R_1\wedge R_2$.
Write $S_3$ for the meet $S_1\wedge S_2$; clearly $S\le S_3$. To fill in the dotted part of the diagram, it remains to show that $\sigma$ defines a  shrink morphism from $(\ord m,u,S_3)$ to $(\ord p,w,R_3)$. We do this by induction on the length of $\sigma$. Choose $j$ minimal with $\sigma(j)=\sigma(j+1)$. By Proposition~\ref{prop:CS-factorization}, we know that $\sigma_j$ defines shrink morphisms $(\ord m,u,S_i)\to(\ord{m-1},u',S'_i)$ for $i=1,2$, and that $\sigma$ factorizes as $\sigma'\sigma_j$ in each case.  By Remark~\ref{rmk:shrink-r}, $S'_i$ can be defined by $r_{S'_i}=\sigma_j r_{S_i} \delta_{j-1}$. By Proposition~\ref{prop:wedge-ells} we know that $\ell_{S_3}(j)=j$, and so that $\sigma_j$ defines a shrink morphism from $(\ord m,u,S_3)$ to $(\ord{m-1},u',S'_3)$, and by Remark~\ref{rmk:shrink-r} once again we have $r_{S'_3}=\sigma_j r_{S_3}\delta_{j-1}$; finally, since $\sigma_j$ preserves meets, it follows that $S'_3=S'_1\wedge S'_2$. By Proposition~\ref{prop:CS-factorization} we know that $\sigma'$ defines a shrink morphism $(\ord{m-1},u',S'_3)\to(\ord p,w,R_3)$, and we may now continue by induction.

We have now constructed all of the dotted maps in the displayed diagram. Commutativity of each individual region follows from the given relations, thus the exterior diagram also commutes.
\end{proof}

\subsection{Skew monoidal structure and universal property}

Since we can tensor generators of each type, it follows that we can tensor arbitrary \F-morphisms, and the tensor product in $\dbot$ lifts to \F. By the faithfulness and strict preservation properties of the forgetful functor to $\dbot$ we conclude that \F is skew monoidal.

\begin{theorem}\label{thm:Fsk}
 % The category of \F-maps
\F is the free
skew monoidal category generated by a single object.
\end{theorem}

\begin{proof}
Let \C be a skew monoidal category and $C$ an 
object of \C. We need to show that there is a unique
strict skew monoidal functor from \F to \C which sends the generator
$X=(1,1,*)$ to $C$. 

By Theorems~\ref{thm:US-univ-prop} and~\ref{thm:UI-univ-prop} we know what $F$ must do on \Fla and \Far, and so on all \F-surjections and \F-injections. Thus the uniqueness part is immediate.
All that remains to show is that if we define $F$ 
on \F-surjections and \F-injections in this way then we
do indeed obtain a functor; the fact that it strictly
preserves the skew monoidal structure is immediate.

Functoriality will follow if we can show that 
$F$ respects each of the relations 
\eqref{eq:UGrelation1}, \eqref{eq:UGrelation2},
\eqref{eq:UGrelation3},
and \eqref{eq:UGrelation4}.

Of these, \eqref{eq:UGrelation1} and \eqref{eq:UGrelation3} are respected thanks to functoriality and 
naturality; we need only worry about
\eqref{eq:UGrelation2} and \eqref{eq:UGrelation4}. But these are respected thanks, respectively, to the commutativity of the diagrams
$$\xymatrix{
& I\ox I \ar[dr]^{\lambda} \\
I \ar[ur]^{\rho} \ar[rr]_{1} && I }\quad
\xymatrix{
(A\ox I)\ox B \ar[r]^{\alpha} & A\ox(I\ox B) \ar[d]^{1\ox\lambda} \\
A\ox B \ar[u]^{\rho\ox 1} \ar[r]_1 & A\ox B}$$
in \C.
\end{proof}

As well as the explicit construction just completed, we also have the following more qualitative result.

\begin{corollary}\label{cor:faithful}
 The skew monoidal category freely generated by a single object admits a unique faithful structure-preserving functor into \dbot which sends the generating object to the unit of \dbot.
\end{corollary}

Since \dbot is a skew monoidal category, if we can find some equation between expressions in the structure morphisms which does not hold in \dbot, then it clearly cannot follow from the skew monoidal category axioms. The corollary means that we have a converse: if such an equation does hold in \dbot then it holds in \Fsk, and so holds in all skew monoidal categories. 

Consider, for example, the idempotent $\epsilon_0$ defined in Section~\ref{sect:axioms}, and given by the composite of $\lambda_I\colon I\ox I\to I$ followed by $\rho_I\colon I\to I\ox I$. In \dbot, this amounts to the unique surjection $\ord 2\to\ord 1$ followed by the injection $\ord 1\to \ord 2$ with image the bottom element of 
\ord 2. This composite is clearly not the identity, and so $\epsilon_0$ need not be the identity; similarly, one sees that $\epsilon^\ell_{X,Y}$ and $\epsilon^r_{X,Y}$ need not be identities.

On the other hand, to see that the composite 
$$\xymatrix{
(W(XI))Y \ar[r]^{(\rho1)1} & ((WI)(XI))Y \ar[r]^{\alpha1} & (W(I(XI)))Y \ar[r]^{\alpha} & W((I(XI))Y) \ar[d]^{1(\lambda1)} \\ &&& W((XI)Y)
~, }$$
in which tensor products have been written as juxtaposition, is equal to $\alpha_{W,XI,Y}$, it suffices to check that this is the case in \dbot. Recalling that in \dbot all instances of $\alpha$ are just the identity, this composite becomes 
$$\xymatrix{
{}\ord 4 \ar[r]^{\delta_1} & \ord 5 \ar[r]^{\sigma_1} & \ord 4 }$$
which is indeed equal to the identity.

\section{The club for skew monoidal categories}
\label{sect:club}

The notion of club \cite{Kelly-AbstractApproachCoherence} was introduced precisely in order to deal with coherence problems for categories with (certain types of) structure. Within that context, the structure of skew monoidal category is of the simplest type:  what was called a {\em club over \NN} in \cite{Kelly-AbstractApproachCoherence}. These clubs over \NN are in fact equivalent to non-symmetric \Cat-enriched operads; this is related to the equivalence between the (monoidal) categories $\Cat^\NN$ of $\NN$-indexed families of categories, and $\Cat/\NN$ of categories equipped with a functor into the discrete category \NN. 

A non-symmetric \Cat-enriched operad  consists of a family $(A_n)_{n\in\NN}$ of categories, indexed by the natural numbers. These are equipped with ``multiplication'' functors
$$A_n\x A_{m_1}\x\ldots\x A_{m_n}\to A_{m_1+\ldots+m_n}$$
satisfying associativity conditions, and an object of $A_1$ serving as a unit. To give a family $(A_n)_{n\in\NN}$ is equivalently to give a category $A$ equipped with a functor into the discrete category \NN; the $A_n$ are then the fibres of the functor. It is this functor $A\to\NN$ which is used in the club point of view; the multiplication and unit can be expressed in terms of $A$ as well. 

In either case, one defines algebras for the club or operad, and the free algebra on a category $X$ is given by the familiar sum $$\sum_{n\in\NN} A_n\x X^n$$
equipped with an algebra structure coming from the multiplication.  In particular, the free algebra on \ord 1 has underlying category $\sum_{n\in\NN} A_n$; in other words, what we were calling $A$.

When we said, in the introduction, that in the club situation the free structure on any category could be obtained from the free structure on \ord 1, we were guilty of a slight oversimplification; what is actually needed is the free structure $A$ on \ord 1, together with the associated functor $A\to\NN$. This is then enough to extract the $A_n$ and so form the sum $\sum_n A_n\x X^n$. While abstractly this ``augmentation'' $A\to\NN$ is extra information, in practice it is generally easy to write down. The objects of $A$ correspond to operations $X^n\to X$ in the structure in questions, and the augmentation merely records the arity (in this case, $n$) of each such operation.

The structures that can be described using clubs over \NN (or equivalently using non-symmetric \Cat-operads) are of the following type. One has a category $X$ equipped with various functors $f\colon X^{n_f}\to X$, with domain $X^{n_f}$ a finite (discrete) power of $X$. These are thought of as operations of arity $n_f$. From these ``basic'' operations, the existence of further operations can be derived: for instance, given $f\colon X^n\to X$ and $g\colon X^m\to X$, we can substitute $g$ into $f$ in (say) the $j$th position, giving a functor 
$$\xymatrix{
X^{n+m-1} \ar[rrr]^{1_{X^{j-1}}\x g\x 1_{X^{m-j}}} && &
X^n \ar[r]^{f} & X. }$$
Then there are natural transformations given between these derived operations. From these ``basic'' natural transformations, further natural transformations can be derived, by composition of transformations, by substituting a natural transformation into an operation, and by substituting an operation into a natural transformation. Finally equations, either between operations or between natural transformations, can be imposed. 

(In this description, we have spoken of operations and of natural transformations, but the natural transformations themselves can usefully be regarded as higher-dimensional operations.)

\begin{example}
In the case of skew monoidal categories, one starts withthe operations $X^2\to X$ and $X^0\to X$ of tensor product and unit. From these one derives two operations $X^3\to X$ using the two possible ternary bracketings, and three operations $X\to X$ given by the identity and tensoring on either side with the unit. The three lax constraints are now natural transformations between these derived operations; finally the five axioms are equations between derived natural transformations. Thus the structure of skew monoidal category can indeed be described by a club over \NN.
\end{example}

After this all-too-scant overview of clubs, we turn to the description of the free skew monoidal category $\Fsk(\XX)$ on a category \XX, using the description of $\Fsk=\Fsk(\ord 1)$ given in the previous section.

First, however, we must describe the augmentation $\Fsk(\ord 1)\to\NN$. This associates to each object $(\ord m,u,S)$ the arity of the corresponding operation. Now $(m,u,S)$ corresponds to an $m$-fold product, bracketed according to $S$, which might sound like an $m$-ary operation. But any $I$'s appearing in the product are actually the output of a nullary operation, so the true arity is the number of $X$'s, which is the cardinality of~$u$.

We can now use the general theory of clubs \cite{Kelly-AbstractApproachCoherence} to write down an explicit description of the free skew monoidal category $\F(\XX)$ on an arbitrary category \XX.

\begin{theorem}
  An object of the free skew monoidal category on a category \XX is an object $(\ord m,u,S)$ of \Fsk, equipped with a $u$-indexed family $X=(X_i)_{i\in u}$ of objects of \XX. A morphism in from $(\ord m,u,S,X)$ to $(\ord n,v,T,Y)$ consists of a morphism $\phi\colon(\ord m,u,S)\to(\ord n,v,T)$ in \Fsk, along with a $u$-indexed family $x=(x_i)_{i\in u}$ of morphisms in \XX, where $x_i\colon X_i\to Y_{\phi i}$. 
\end{theorem}

\enlargethispage{\baselineskip}
\enlargethispage{\baselineskip}

\bibliographystyle{amsplain}
%\bibliography{my}

\end{document}